\documentclass[11pt,reqno,letter]{amsart}

\usepackage{graphicx} % Required for inserting images
\usepackage{amsmath,amsthm,verbatim,amscd,amssymb,setspace,enumitem,hyperref,exscale,color,slashed, arydshln, accents, tensor}
\usepackage[all,pdf]{xy}

\newcommand{\CC}{\mathbb{C}}
\newcommand{\RR}{\mathbb{R}}

\newcommand{\OO}{\mathcal{O}}

\newcommand{\K}{\mathcal{K}}

\newcommand{\C}{\mathcal{C}}    
\newcommand{\HH}{\mathcal{H}}

\newcommand{\DD}{\mathcal{D}}

\newcommand{\LL}{\mathcal{L}}

\newcommand{\m}{\mathfrak{m}}
\newcommand{\PP}{\mathbb{P}}
\newcommand{\ZZ}{\mathbb{Z}}
\newcommand{\Ric}{\textup{Ric}}

\newcommand{\re}{\operatorname{Re}}
\newcommand{\tr}{\operatorname{tr}}

\newcommand{\dbtilde}[1]{\accentset{\approx}{#1}}

\newtheorem{thm}{Theorem}[section]

\newtheorem{pro}[thm]{Proposition}

\newtheorem{lem}[thm]{Lemma}

\newtheorem*{qs}{Question}

\newcounter{mtheorem}
\newtheorem{mtheorem}[mtheorem]{Theorem}

\theoremstyle{definition}
\newtheorem{rem}[thm]{Remark}

\numberwithin{equation}{section}

\title{Geodesic Equations on asymptotically locally Euclidean K\"ahler manifolds}
\author{Qi Yao}
\date{}

\begin{document}

\maketitle

\begin{abstract}
In this paper, We solve the geodesic equation in the space of Kähler metrics under the setting of asymptotically locally Euclidean (ALE) K\"ahler manifolds and establish global $\C^{1,1}$ regularity of the solution. The solution of the geodesic equation is then related to the uniqueness of scalar-flat ALE metrics. To this end, we study the asymptotic behavior
of $\varepsilon$-geodesics at spatial infinity. We will prove the convexity of Mabuchi $K$ energy along $\varepsilon$-geodesics under the assumption that the Ricci
curvature of a reference ALE K\"ahler metric is non-positive. However, by testing the Ricci curvature of ALE K\"ahler metrics, we find that on the line bundle $\OO(-k)$ over $\CC\PP^{n-1}$ with $n \geq 2$ and $k \neq n$, all ALE Kähler metrics cannot have non-positive (or non-negative) Ricci curvature.
\end{abstract}

\section{Introduction} 

In the paper, we study the geodesic equation in the setting of  ALE K\"ahler cases, assuming relatively weak fall-off conditions. Let $(X, J, g)$ be a complete non-compact K\"ahler manifold of complex dimension $n$ ($n\geq 2$), we say $(X, J, g)$ is ALE if there is a compact subset $K \subseteq X$ such that $ \psi: X\setminus K \rightarrow (\CC^n \setminus B_R)/ \Gamma$ is a diffeomorphism, where $B_R$ is a closed ball in $\CC^n$ with radius $R$ and $\Gamma$ is a finite subgroup of $U(n)$ (any ALE K\"ahler manifold has only one end according to \cite[Proposition 1.5, 3.2]{hein2016mass}) and the metric $g$ satisfies the following condition on the end $X\setminus K$:
\begin{enumerate}
 \item[$\bullet$] The metric $g$ is asymptotic to the Euclidean metric $\delta_{ij}$ at the end with decay rate $-\tau$ for some $\tau> n-1$, i.e., for $i =0,1,\ldots, k$,
\begin{align} \label{decayale}
g_{ij}= \delta_{ij} + O(r^{-\tau}) , \qquad  |\nabla^i  ((\psi^{-1})^* g) |_{g_0} = O(r^{-\tau-i}).
\end{align}
\end{enumerate}
The fall-off condition $\tau > n-1 $ is the weakest decay rate to make the ADM mass coordinate-invariant in general, referring to Bartnik \cite{bartnik1986mass} and Chru\`{s}ciel \cite{Chruciel1985BoundaryCA}. 

One of the difficulties in building up a general theory of scalar-flat K\"ahler metrics in the ALE setting is that the decay rate of such metrics to their asymptotic models is not good enough compared to the Ricci-flat case. For instance, consider the family of scalar-flat K\"ahler metric constructed on $\OO_{\CC\PP^1}(-k)$ by LeBrun \cite{cmp/1104162166}, 
\begin{align*}
    g = \frac{ds^2}{1+ A/s^2 +B /s^4} + s^2 \Big[\sigma_1^2 + \sigma_2^2 + \Big(1+ \frac{A}{s^2} + \frac{B}{s^4}\Big) \sigma_3^2\Big],
\end{align*}
where $A$, $B$ are constants, $\sigma_{1}$, $\sigma_{2}$, $\sigma_3$ are three invariant vector fields on $3$-sphere and $s$ is a radial function on $\OO(-k)$. It can be checked that $g-g_{euc} = O(r^{-2})$, where $r$ denotes the geodesic distance from a fixed basepoint, indicating that the Kähler potential function should be of $\log$ growth. In Arezzo-Pacard \cite[Lemma 7.2]{arezzo2006blowing}, an expansion theorem is proved for scalar-flat K\"ahler metrics in the complement of $B_\Gamma = \{ z \in \CC^n / \Gamma : |z| \leq 1 \}$ in $\CC^n / \Gamma$, where $\Gamma$ is a finite subgroup of $U(n)$, assuming that the $dd^c$-lemma holds in this situation. In \cite{yao2022mass}, the author proved a $dd^c$ lemma and an expansion theorem under the setting of asymptotically conical (AC) K\"ahler manifolds. Here, we only need a theorem of weaker version under the setting of ALE K\"ahler manifolds.

\begin{thm} \label{yao2022} (Yao 2022)
Let $(X, J)$ be an ALE K\"ahler manifold asymptotic to $\CC^n / \Gamma$.  Let $\omega_1$, $\omega_2$ be K\"ahler forms in the same Kähler class of $(X, J)$ with the corresponding metrics satisfying \eqref{decayale} and such that the scalar curvatures of $\omega_1$ and $\omega_2$ are equal, $R_1 \equiv R_2$. 
Then  
\begin{align}
\omega_2 = \omega_1 + dd^c \varphi, \quad \text{ with the potential } 
\varphi \in \C^{\infty}_{2-2\tilde{\tau}}
\end{align}
for some $\tilde{\tau} >n-1$ depending on $(n, \tau)$.
\end{thm}

Let $\omega$ be the corresponding K\"ahler form of $g$.  
According to Theorem \ref{yao2022}, given two K\"ahler forms $\omega_1,\omega_2 \in [\omega]$, if the corresponding ALE K\"ahler metrics $g_1,g_2$ satisfy the decay condition \eqref{decayale} and that the scalar curvatures of $g_1$ and $g_2$ are identically equal, $R(g_1)\equiv R(g_2)$, then $\omega_1 - \omega_2 = dd^c f$ and $f$ decays at infinity with higher rate $-\gamma$, with $\gamma = 2\tilde{\tau}-2$, for some $\tilde{\tau} > n-1$.
Hence, for the prescribed scalar curvature problem, we consider the following restricted weighted K\"ahler potential space,
\begin{align*}
    \HH_{-\gamma} (\omega) = \{ \varphi \in \hat{\C}^{\infty}_{-\gamma}: \omega_\varphi = \omega + dd^c\varphi > 0\} \quad (\gamma > 2n-4\geq0),
\end{align*}
where the class of functions, $\hat{\C}^\infty_s$, is defined as follows 
\begin{align*}
    \C^\infty_{s} &= \{ f\in \C^\infty (X) : |\nabla_{g_0}^j f|_{g_0} = O(r^{s-j})\;\,\text{for all} \;\, j \geq 0 \},\\
    \hat{\C}^\infty_{s} &= \{\hat{f}\in \C^\infty(X): \hat{f} = f + c, \text{ for } f\in \C^{\infty}_{s} \text{ and } c \text{ is a constant}\}.
\end{align*}

Define $\omega_0 = \omega + dd^c \psi_0$, $\omega_1 = \omega+ dd^c \psi_1$, for any two boundary data $\psi_{0,1} \in \HH_{-\gamma}(\omega)$. Also introduce the linear reference path $\psi(t) = (1-t)\psi_0 + t \psi_1$ in $\HH_{-\gamma}(\omega)$. Another path $\varphi(t)$ in $\HH_{-\gamma}(\omega)$ with the same endpoints $\psi_0,\psi_1$ is called a geodesic in $\HH_{-\gamma}(\omega)$ if 
\begin{align} \label{geoequ1}
    \ddot{\varphi}(t) - \frac{1}{2} |\nabla_{\omega_{\varphi(t)}} \dot{\varphi}(t)|_{\omega_{\varphi(t)}}^2 = 0.
\end{align}
As observed by Donaldson \cite{donaldson1999symmetric} and Semmes \cite{semmes1992complex}, the geodesic equation is equivalent to a homogeneous complex Monge-Amp\`{e}re equation in the product space $X \times \Sigma$, where $\Sigma \cong [0,1]\times S^1$ can be embedded as an annulus in $\CC$. Notice that any path $\varphi(t)$ of functions on $X$ can be viewed as a function $\Phi$ on $X \times \Sigma$ via $\Phi (\cdot, t, e^{is}) = \varphi(t)$. Let $\Omega_{\Phi} = p^* \omega + dd^c \Phi$, where $p$ is the projection from $X\times \Sigma$ to $X$ and $dd^c\Phi$ is computed on $X \times \Sigma$. Then the equation \eqref{geoequ1} can be rewritten as follows:
\begin{align}
    &\Omega_{\Phi}^{n+1} = 0, \label{geoequ2} \\
    &\Omega_{\Phi} \geq 0, \label{geoequ2p}\\
    &\Phi|_{t=0,1} = \psi_{0,1}. \label{geoequ2b}
\end{align}

In \cite{donaldson1999symmetric}, Donaldson proposed a program to attack the existence and uniqueness problems regarding canonical metrics by studying the geometric structure of the potential space $\HH$, where the geodesic equation plays a central role. In the cases of compact K\"ahler manifolds, Chen \cite{chen2000space} showed that for any $\psi_0$, $\psi_1 \in \HH$, the geodesic equation has a unique solution up to $dd^c$-regularity. Blocki \cite{blocki2009gradient} and He \cite{he2015space} built up direct calculations to prove the gradient estimate and Laplacian estimate. The full $\C^{1,1}$ estimate was proved by Chu-Tosatti-Weinkove in \cite{chu2017regularity}. In the other direction, Lempert-Vivas \cite{lempert2013geodesics} and Darvas-Lempert \cite{darvas2012weak} constructed counter-examples to assert that $dd^c \Psi$  is not continuous in general, hence the $\C^{1,1}$ regularity is optimal in general. In \cite{auvray2017space}, Auvray generalized the $dd^c$-regularity to singular cases (precisely, there exist cusp singularities along simple normal crossings). The main theorem of sections \ref{sectionopen}-\ref{sectionC11} is to generalize the full $\C^{1,1}$ estimates to ALE K\"ahler manifolds.

\begin{mtheorem}\label{muniformc11}
Let $X$ be an ALE K\"ahler manifold and $\psi_0,\psi_1 \in \HH_{-\gamma} (\omega)$ $(\gamma > 0)$. Then $\psi_0$ and $\psi_1$ can be connected by a $\C^{1,1}$ geodesic $\Phi$ solving \eqref{geoequ2}, \eqref{geoequ2p}, \eqref{geoequ2b}. Moreover, there is a uniform constant $C$ depending only on $\|\psi_0\|_{\C^{1,1}(X,\omega)}$, $\|\psi_1\|_{\C^{1,1} (X,\omega)}$ and on the geometry of $(X, \omega)$ such that
\begin{align}
\sup_{X \times \Sigma}\big( |\Phi| + |\nabla_{\Theta_\Psi} \Phi|_{\Theta_\Psi} + |{\nabla}^2_{\Theta_\Psi} \Phi|_{\Theta_\Psi}  \big)  \leq C.
\end{align}
Here, $\Theta_\Psi$ is a Kähler form on $X \times \Sigma$ given by $\Theta_\Psi = \Theta + dd^c \Psi$ with $\Psi(\cdot,t,e^{is}) = \psi(t) = (1-t)\psi_0 + t\psi_1$ the linear path introduced above, and with $\Theta = p^*\omega + A dd^c t(t-1)$, where $A>0$ is fixed depending only on $\|\psi_0\|_{\C^{1,1}(X,\omega)}$, $\|\psi_1\|_{\C^{1,1} (X,\omega)}$ such that $\Theta_\Psi > 0$.
\end{mtheorem}

Then, we relate the solution of the geodesic equation to the uniqueness of scalar-flat ALE K\"ahler metrics in each K\"ahler class. The main idea is to follow the framework of Chen \cite{chen2000space} in the compact case, under the assumption that the Ricci curvature of the reference metric is non-positive. This was extended to the non-compact case with Poincar\'e cusp ends by Auvray \cite{auvray2017space}. In the ALE case, it is first necessary to prove sufficient decay at infinity of solutions to the $\varepsilon$-geodesic equation.

In Section \ref{secasymbehavior}, we discuss the asymptotic behavior of $\varepsilon$-geodesics. Given any two functions
\begin{align*}
    \psi_0,\psi_1 \in \HH_{-\gamma} (\omega) = \{ \varphi  \in \hat{\C}^{\infty}_{-\gamma}: \omega_\varphi = \omega + dd^c\varphi > 0 \} \quad (\gamma > 0),
\end{align*}
we set $\psi(t) = (1-t)\psi_0 + t\psi_1$ and let $\Psi$ denote the corresponding function on $X \times \Sigma$. We fix $A$ large depending on $\|\psi_0\|_{\mathcal{C}^{1,1}(X,\omega)}$, $\|\psi_1\|_{\mathcal{C}^{1,1}(X,\omega)}$ such that $\Theta_\Psi := \Theta + dd^c \Psi$ is positive on $X \times \Sigma$, where $\Theta := p^*\omega + A dd^c t(t-1)$ with $p: X \times \Sigma \to X$ the projection.  Then, we introduce the following $\varepsilon$-geodesic equations 
\begin{align*}
    (E_{\varepsilon}) \qquad 
    \begin{cases}
    (\Theta + dd^c \Phi_\varepsilon )^{n+1} = \upsilon(\varepsilon) \Theta_\Psi^{n+1}, \quad &\text{ in } X \times \Sigma,\\
    \Theta + dd^c \Phi_\varepsilon >0, & \text{ in } X \times \Sigma,\\
     \Phi_\varepsilon|_{t =0,1} = \phi_{0,1}, & \text{ on } X \times \partial \Sigma,
    \end{cases}
\end{align*}
where $\upsilon(\varepsilon)$ is a smooth nonnegative function defined in $X \times \Sigma \times [0,1]$ satisfies the following conditions 
\begin{align} \label{conupsilon}
\begin{split}
  &\upsilon(0) \equiv 0, \quad   \upsilon(1) \equiv 1;\\
  &\upsilon(\varepsilon) > 0,\qquad \qquad \qquad \qquad \  \text{ for } \varepsilon \in (0,1];\\
  & C^{-1} \varepsilon \leq \upsilon(\varepsilon) \leq \min ( C \varepsilon , 1 ), \quad  \text{for } \varepsilon \in [0,1];\\
  &|\nabla^k \upsilon(y, \varepsilon)| \leq C \varepsilon r(y)^{- \varsigma -k}, \quad \text{ for } (y, \varepsilon) \in X \times \Sigma \times [0,1], \quad k\geq 1
\end{split}
\end{align}  
where $\varsigma$ is an real number with $\varsigma \geq \gamma$. 
In particular, we are interested in the following case.
By taking 
\begin{align} \label{coeffEvarepsilon}
\upsilon(\varepsilon)= \varepsilon ((1-\chi(\varepsilon))f + \chi(\varepsilon)),
\end{align}
where $\chi$ is a smooth increasing function in $[0,1]$ equal to $0$ (resp. $1$) in a neighborhood of $0$ (resp. $1$) and $f$ is defined as follows
\begin{align} \label{coeffEvarepsilonf}
f = A^{-1}\frac{\Theta^{n+1}}{\Theta_\psi^{n+1}} \in \C^{\infty} (X \times \Sigma),
\end{align}
and in this case, $|\nabla^k f| \leq C r^{-\gamma-2-k} $, $\varsigma = \gamma+2$.
By taking $\varepsilon$ to be small enough, $(E_\varepsilon)$ can be written as
\begin{align*}
     \Big(\ddot{\varphi} - \frac{1}{2} |\nabla_{\omega_{\varphi}} \dot{\varphi}|_{\omega_{\varphi}}^2 \Big)  \omega_{\varphi}^{n} = \varepsilon \omega^n.
\end{align*}
Due to the positivity of the right hand side of $(E_\varepsilon)$, it is well known that for every $\varepsilon \in (0,1]$ there exists a solution $\Phi_\varepsilon \in \bigcap_{k,\alpha} \mathcal{C}^{k,\alpha}$. We now prove:

\begin{mtheorem}\label{epsilongeodesicasymptotics}
Let $\Phi_\varepsilon$ be the $\varepsilon$-geodesic constructed above. Then, there exists a constant $C(k,\varepsilon^{-1})$ depending on $k \geq 1$ and on an upper bound for $\varepsilon^{-1}$ such that
    \begin{align*}
        \left(|\nabla^{k}_{X,\omega}\Phi_\varepsilon|_\omega  + |\nabla^{k}_{X,\omega} \dot{\Phi}_\varepsilon|_\omega +  |\nabla^k_{X,\omega} \ddot{\Phi}_\varepsilon|_\omega\right) \leq C(k,\varepsilon^{-1}) r^{-\gamma -k}\quad\text{for all}\ k \geq 1,
    \end{align*}
    where $\nabla_{X,\omega}$ denotes the Levi-Civita connection of the ALE Kähler metric $\omega$ on $X$, acting as a differential operator in the $X$ directions on $X \times \Sigma$. And 
    \begin{align*}
        |\Phi_\varepsilon - c(t)| \leq C(\varepsilon^{-1})r^{-\gamma}, 
    \end{align*}
    where $c(t)$ is a function only depending on $t$.
    Hence, for any two potentials $\psi_0$, $\psi_1$ in $\HH_{-\gamma} (\omega)$, there exist $\varepsilon$-geodesics in $\HH_{-\gamma}(\omega)$ connecting $\psi_0$ and $\psi_1$.
\end{mtheorem}

In section \ref{secasymbehavior}, we prove a stronger statement. Let $\varphi_{\varepsilon} = \Phi_{\varepsilon}- \Psi$, then $\varphi_{\varepsilon} \in \HH_{\max\{-2\gamma-2, -\varsigma\}}$ due to the fact that $\Psi$ was chosen to be linear in $t$ (see section \ref{secasymbehavior} for details).

Hence, while we still cannot define the Mabuchi $K$-energy along geodesics, the Mabuchi $K$-energy is now actually well-defined along $\varepsilon$-geodesics assuming $\gamma = 2\tilde{\tau}-2 > 2n-4$. 

In Section \ref{secconvK}, the second derivative of the Mabuchi $K$-energy will be calculated. Throughout section \ref{secconvK}, we assume $\gamma = 2\tilde{\tau}-2 > 2n-4$ (Here it turns out that if $\psi_0$, $\psi_1$ are only in $\HH_{4-2n} (\omega)$, there would be boundary terms at infinity breaking the positivity of the second derivative. This is a new phenomenon compared to Chen \cite{chen2000space} and Auvray \cite{auvray2017space}). However, under the assumption that the Ricci curvature of some reference ALE K\"ahler metric, $\omega$, is non-positive, we can then prove the convexity of Mabuchi $K$-energy:

\begin{mtheorem}\label{partialunique}
    Assume that $\omega$ is an ALE K\"ahler metric on $X$ such that the Ricci curvature of $\omega$ is non-positive, $\Ric (\omega) \leq 0$. Then, along each $\varepsilon$-geodesic in $\HH_{-\gamma}(\omega)$ with $\gamma>2n-4$, $\varphi(t)$, the Mabuchi $K$-energy is convex.
\end{mtheorem}

A quick corollary of Theorem \ref{partialunique} is that assuming $\Ric(\omega) \leq 0$, the scalar-flat K\"ahler metric, if it exists, is unique in $\HH_{-\gamma} (\omega)$. However, if there exists a scalar-flat K\"ahler metric $\omega_0$ in $\HH_{-\gamma} (\omega)$, the condition, $\Ric (\omega) \leq 0$ implies $\Ric (\omega) =0$. Hence, the uniqueness of scalar-flat ALE metric can be reduced to the uniqueness result of Ricci-flat ALE K\"ahle metric, which can be found in reference \cite{joyce2000compact, vancoevering2010regularity, conlon2013asymptotically}. The point is that $\omega_0 = \omega + O(r^{-\gamma-2})$ implies by definition that the ADM masses of $\omega$ and $\omega_0$ are equal, $\m(\omega) = \m (\omega_0)$. According to mass formula by Hein-LeBrun \cite{hein2016mass}, it follows that $\int R(\omega) = \int R (\omega_0) =0$. The assumption that $\Ric (\omega) \leq 0$ implies that $\Ric (\omega) =0$ (see Remark \ref{remunique} for details). In fact, in Section \ref{secnonexist}, we will prove that many ALE K\"ahler manifolds do not admit any ALE K\"ahler metrics with $\Ric \leq 0$ (or $\Ric \geq 0$) at all:

\begin{mtheorem}\label{nonex}
    Let $\OO(-k)$ be the standard negative line bundle over $\CC\PP^{n-1}$ with $n \geq 2$, $k \ne n$, 
    and let $\omega$ be an ALE K\"ahler metric on $\OO(-k)$ with decay rate $-\tau$, $\tau > 0$. Then, the Ricci form of $\omega$, is of mixed type, i.e., neither $\Ric (\omega) \geq 0$ nor $\Ric (\omega) \leq 0$ is true.
\end{mtheorem}

In Riemannian geometry, AE metrics of negative Ricci curvature are well-known to exist in $\RR^n$  by explicit construction in Lohkamp \cite{Lohkamp1994}. Theorem \ref{nonex} gives a negative answer to this question in the setting of ALE K\"ahler metrics.

An interesting question in this context is to ask whether some version of the Nonexistence Theorem \ref{nonex} holds in general ALE K\"ahler manifolds or even AC K\"ahler manifolds.

\begin{qs}
    Is it true in any ALE K\"ahler manifold that the Ricci curvature form of an ALE K\"ahler metric can only be identically zero or of mixed type?
\end{qs}

This paper is a part of the Ph.D. thesis of the author. The author would like to express his gratitude to Professor Hans-Joachim Hein and Professor Bianca Santoro for suggesting the problem, and for constant support, many helpful comments, as well as much enlightening conversation. The author is also thankful to professor Gustav Holzegel for providing financial support via his Alexander von Humboldt Professorship during the last semester at University of Münster.
The whole project is Funded by the DFG under Germany's Excellence Strategy EXC 2044-390685587, Mathematics Münster: Dynamics-Geometry-Structure, and by the CRC 1442, Geometry: Deformations and Rigidity, of the DFG.

\section{$\varepsilon$-geodesic equations and openness} \label{sectionopen}
Recall  that $\varepsilon$ geodesic equations can be written as follows, 
\begin{align*}
    (E_{\varepsilon}) \qquad 
    \begin{cases}
    (\Theta + dd^c \widetilde{\Phi} )^{n+1} = \upsilon(\varepsilon) (\Theta+ dd^c \Psi)^{n+1}, \quad &\text{ in } X \times \Sigma,\\
    \lambda\Theta < \Theta + dd^c \widetilde{\Phi} < \Lambda \Theta, & \text{ in } X \times \Sigma,\\
     \widetilde{\Phi}|_{t =0,1} = \psi_{0,1}, & \text{ on } X \times \partial \Sigma,
    \end{cases}
\end{align*}
where $\varepsilon \in (0,1]$ and $0 < \lambda < \Lambda$ are constants depending on $\varepsilon$. The family of equations $(E_\varepsilon)$ is called the $\varepsilon$-geodesic equations. 
The idea to solve the equation $(E_\varepsilon)$ is the following. Firstly, we apply the continuity method to show that there exists a solution of $(E_\varepsilon)$ in $\C^{k,\alpha}$. In particular, consider the family of equations $(E_{s})$, $s \in [\varepsilon,1]$. There is a trivial solution at $(E_1)$. Then, we shall prove the openness and closedness of $(E_s)$ in certain regularity. In the current section, we deal with the openness of $(E_s)$.

Assuming that there exists a solution of $(E_{s_0})$ in $\C^{k,\alpha}$ for some $s_0 \in [\varepsilon,1]$, we will show in this subsection that $(E_s)$ can be solved for all $s$ in a small open neighborhood of $s_0$. For simplicity, we write $\Theta_\Psi = \Theta +dd^c \Psi$ as in Theorem \ref{muniformc11} and $\widetilde{\varphi} = \widetilde{\Phi}- \Psi$. Then, the equation $ (E_\varepsilon)$ can be written as, $ (\Theta_\Psi + dd^c \widetilde{\varphi})^{n+1} = \varepsilon \Theta_{\Psi}^{n+1}$ in $X\times \Sigma$,
with the boundary condition $\widetilde\varphi =0$ on $X \times \partial \Sigma$. Then, the Monge-Amp\`{e}re operator is defined to be
\begin{align*}
    \mathcal{M} (\chi) = \frac{(\Theta_\Psi + dd^c \chi)^{n+1}}{\Theta_{\Psi}^{n+1}}.
\end{align*}
Let $\widetilde\varphi$ be a solution of $(E_{s_0})$ for some $s_0 \in [\varepsilon, 1]$. By assumption, $\widetilde\varphi$ is $\Theta_\Psi$-plurisubharmonic satisfying $c\Theta \leq \Theta_\Psi + dd^c \widetilde\varphi \leq C \Theta$. Then, the linearization of Monge-Amp\`{e}re operator at $\widetilde\varphi$ is uniformly elliptic, and is given by
\begin{align*}
    \LL_{\widetilde\varphi} (\chi ) =\big( \Delta_{{\widetilde\varphi}} \chi \big) \cdot \frac{(\Theta_\Psi + dd^c \widetilde\varphi)^{n+1}}{\Theta_{\Psi}^{n+1}} = s_0 \Delta_{\widetilde\varphi} \chi,
\end{align*}
where $\Delta_{\widetilde\varphi}$ represents the Laplacian with respect to $\Theta_\Psi + dd^c \widetilde\varphi$. Let $(\C^{k,\alpha})_0$ be the functions in $\C^{k,\alpha}$ vanishing on the boundary $X \times \partial \Sigma$. Then, we have the following property of $\LL_{\widetilde\varphi}$, from which the desired openness is clear by the implicit function theorem.

\begin{pro} \label{isolinMA}
Let $\widetilde\varphi$ be the solution of $(E_{s_0})$, then the linearized operator $\LL_{\widetilde\varphi} : ({\C}^{k,\alpha})_0 \rightarrow {\C}^{k-2,\alpha}$ is an isomorphism for all integers $k \geq 2$ and $\alpha\in (0,1)$.
\end{pro}

\begin{proof}
Let us first prove the surjectivity. Fixing $f \in {\C}^{k-2,\alpha}$, the exhaustion argument will be applied to solve the equation $\LL_{\widetilde\varphi} u =f$. Take an exhaustive sequence of pre-compact sets, $\Omega_k \subseteq X \times \Sigma$, with smooth boundary. In particular, by taking a sequence of subsets, $B_{r_k} \times \Sigma$ where $B_{r_k} = \{x\in X:r(x) \leq r_k\}$, and smoothing the corners, we can obtain the exhaustive sequence $\{\Omega_k\}$. Then, we can solve the following Dirichlet problems,
\begin{align*}
    (L_k) \quad \begin{cases}
    \LL_{\widetilde\varphi} u_k = f \quad & \text{ in } \Omega_k,\\
    u_k = 0 & \text{ on } \partial \Omega_k,
    \end{cases}
\end{align*}
where $f \in {\C}^{k-2, \alpha}$.
The existence of the solution of $(L_k)$ is a classic result of the Dirichlet problem on compact Riemannian manifolds with boundary. The key to complete the proof is to give the uniform estimates of $u_k$. The main idea to show the $\C^0$ uniform estimates is to construct barrier functions. Consider the function $A t(1-t) $. The fact that $\lambda \Theta \leq \Theta_\Psi + dd^c \widetilde\varphi \leq \Lambda \Theta $ implies $\Delta_{\widetilde\varphi}  A t(1-t) \leq -\lambda A$. If we suppose that $\|f\|_{L^\infty} \leq C_0$ and take $A = C_0/\lambda$, then we have $\Delta_{\widetilde\varphi} A t(1-t) \leq f = \Delta_{\widetilde\varphi} u_k$. Combining with the fact that $ A t (1-t ) \geq 0$ on the boundary $\partial \Omega_k $, the maximum principle implies that,
\begin{align} \label{linc0}
    \|u_k\|_{L^\infty} \leq \frac{C_0}{\lambda} t(1-t) \leq \frac{C_0}{4\lambda}.
\end{align}
The uniform $\C^{k,\alpha}$ estimates follows directly from the standard Schauder estimates. Precisely, for interior points  $p \in \Omega_k$ away from the boundary, we pick a pair of balls centered at $p$, $ B_{\frac{1}{4}} (p) \subseteq B_{\frac{1}{2}} (p) \subset \Omega_k $. Then, the interior Schauder estimates implies that $\|u_k\|_{k,\alpha; B_{\frac{1}{16}}(p)} \leq C (\|u_k \|_{L^\infty (B_{\frac{1}{8}}(p))} + \|f\|_{k-2, \alpha; B_{\frac{1}{8}}(p)})$. If $p \in \Omega_k$ is close to the boundary, we can apply the boundary Schauder estimate. After straightening the boundary in case the boundary portion on $\partial \Omega_k$ is not flat, we can pick half balls, $p \in B^+_{\frac{1}{4}}(q) \subseteq B^+_{\frac{1}{2}}(q) $ for some $q \in \partial \Omega_k$. Together with the interior estimates, we have
\begin{align} \label{linck}
    \|u_k\|_{k,\alpha; \Omega_k} \leq C(\|u_k \|_{L^\infty(X\times \Sigma)} + \|f\|_{k-2,\alpha; X\times \Sigma}),
\end{align}
where $C$ depends only on $n,k,\alpha,\lambda,\Lambda$. After passing to a subsequence, we conclude that the limit function, $u$, satisfies $\LL_{\widetilde\varphi} u =f $ in $X \times \Sigma$ and $u \equiv 0 $ on $X \times \partial \Sigma$. The uniqueness directly follows from the following maximum principle, Lemma \ref{openmax2}. 
\end{proof}

The following lemma comes from Yau's generalized maximum principle, referring to \cite{cheng1980existence, wu2008kahler}. To describe the model metric on $X \times \Sigma$, we introduce the asymptotic coordinates of $ X \times \Sigma$. Let $\{z_1, \ldots, z_n\}$ be asymptotic coordinates of the end of $X$ and let $w= t+is$ be the complex coordinate of $\Sigma$. Real asymptotic coordinates are given by $\{x_1, \ldots, x_{2n}, x_{2n+1} = t, x_{2n+2} = s\}$, where the complex coordinates are written as $z_{i} = x_{2i-1} + ix_{2i}$. The asymptotic coordinate system will be applied to describe the asymptotic behavior of prescribed K\"ahler metrics on $X\times \Sigma$.

\begin{lem} \label{openmax1} Let $(X\times \Sigma, \Theta_{\widetilde\Phi})$ be the noncompact K\"ahler manifold as above with the K\"ahler metric $\widetilde{g}$ associated with $\widetilde\Phi$ satisfying, for some uniform constant $0 < \lambda < \Lambda$,
\begin{align*}
    \lambda\delta_{ij} \leq \widetilde{g}_{ij} \leq \Lambda \delta_{ij}
\end{align*}
in the asymptotic coordinates of $X\times \Sigma$. Let $u$ be a  $\C_{loc}^2$ function bounded from above on $X \times \Sigma$. Suppose that $\sup_{X \times \Sigma}u > \sup_{X \times \partial \Sigma} u$, then there exists a sequence $\{x_k\}$ in $X \times \Sigma^\circ$ such that
\begin{align} \label{openmax1inq}
    \lim_{k \rightarrow \infty} u(x_k) = \sup_{X\times \Sigma} u, \quad  \lim_{k \rightarrow \infty}|du(x_k)|_{\widetilde{g}} = 0, \quad \limsup_{k\rightarrow \infty} \Delta_{\widetilde{g}} u (x_k) \leq 0.
\end{align}
\end{lem}

\begin{proof} Let $r$ be the radial function inherited from the asymptotic chart of $X$, for instance, $r = (\sum_{i=1}^n|z_i|^2)^{1/2}$. The radial function can be extended to a non-negative smooth function in the whole space $X \times \Sigma$ satisfying the estimate
\begin{align} \label{radialf}
    |\nabla_{\widetilde{g}} r|_{\widetilde{g}} \leq C, \qquad  |\Delta_{\widetilde{g}} r|\leq C,
\end{align}
for some uniform constant $C$. Consider the function $u_\mathbf{e} = u- \mathbf{e} r$. Since $u_\mathbf{e}$ tends to negative infinity as $r$ goes to infinity, $u_\mathbf{e}$ achieves its maximum at some point $x_\mathbf{e}$. And $x_\mathbf{e}$ must be an interior point in $X \times \Sigma$ based on the assumption that $\sup_{X \times \Sigma} u > \sup_{X \times \partial \Sigma} u$. 
At $x_\mathbf{e}$, the function $u_\mathbf{e}$ satisfies 
\begin{gather*}
   0=  d u_\mathbf{e} (x_\mathbf{e}) = d u (x_\mathbf{e}) - \mathbf{e} d r (x_\mathbf{e}),\\
   0 \geq  \Delta_{\widetilde{g}} u_\mathbf{e} (x_\mathbf{e}) = \Delta_{\widetilde{g}} u(x_\mathbf{e}) - \mathbf{e} \Delta_{\widetilde{g}} r (x_\mathbf{e})
\end{gather*}
and 
\begin{align*}
    u_{\mathbf{e}} (x_\mathbf{e}) \geq u(x) - \mathbf{e} r(x),\quad \text{ for all } x \in X \times \Sigma.
\end{align*}
 Choosing $\{x_k\}$ to be points achieving the maximum of $u_{1/k}$, then combining with \eqref{radialf} and letting $k$ go to infinity, we complete the proof of \eqref{openmax1inq}.
\end{proof}

The following lemma is a strengthened version of the above maximum principle, based on solving the Dirichlet problem in $X \times \Sigma$.

\begin{lem} \label{openmax2}
Let $(X\times \Sigma, \widetilde{g})$ be the same as in Lemma \ref{openmax1}. Suppose that $u$ is a function in $\C^{2}_{loc} (X\times \Sigma)$ and bounded from above. Suppose that $u$ satisfies $\Delta_{\widetilde{g}} u \geq 0$ in $X \times \Sigma$ and $ u \leq 0 $ on $X \times \partial\Sigma$. Then $u \leq 0$ in $X\times \Sigma$.
\end{lem}

\begin{proof} Assuming $u$ satisfies $\sup_{X\times \Sigma} u \geq \delta>0$. According to the surjectivity part of the proof of Proposition \ref{isolinMA}, there exists a function $v $ satisfying
\begin{align*}
    \begin{cases}
    \Delta_{\widetilde{g}} v =-1, \quad  &\text{ in } X \times \Sigma,\\
    v = 0, & \text{ on } X \times \partial \Sigma,
    \end{cases}
\end{align*}
and $\|v\|_{L^\infty} \leq C (n, \lambda, \Lambda)$. Consider the function $u_{\mathbf{e}} = u -\mathbf{e} v $ for $\displaystyle \mathbf{e} = \frac{\delta}{2C} $. Then $\displaystyle \sup_{X \times \Sigma} u_\mathbf{e} \geq \frac{\delta}{2} >0$ and $\displaystyle \Delta_{\widetilde{g}} u_\mathbf{e} \geq \mathbf{e}$. According to Lemma \ref{openmax1}, there exists a sequence $\{x_k\}$ in $X \times \Sigma^\circ $ such that $\lim_{k\rightarrow \infty} u_\mathbf{e} (x_k) = \sup_{X \times \Sigma}u_\mathbf{e}$, $\lim_{k\rightarrow \infty}|du_\mathbf{e}(x_k)|_{\widetilde{g}} =0$, $\limsup_{k \rightarrow \infty} \Delta_{\widetilde{g}} u_\mathbf{e}(x_k) \leq 0$. However, $\Delta_{\widetilde{g}} u_\mathbf{e} >0$, which leads to the contradiction.
\end{proof}

\section{A priori estimate up to $\mathcal{C}^0$} \label{sectionC0}

From section \ref{sectionC0} to \ref{sectionC11}, we complete the proof of Theorem \ref{muniformc11}. The key ingredient is to prove uniform a priori estimates up to order $\mathcal{C}^{1,1}$ for the solution $\widetilde\varphi = \widetilde{\Phi} - \Psi$ of the $\varepsilon$-geodesic equation ($E_\varepsilon$). These estimates will be uniform with respect to $\varepsilon \in (0,1]$ and with respect to the distance from a fixed point in $X$. (In section \ref{secasymbehavior}, we will also see that for a fixed $\varepsilon > 0$ it can be proved that $\widetilde\varphi$ is decaying at spatial infinity. However, we are currently unable to make these decay estimates uniform with respect to $\varepsilon$.)

These uniform $\mathcal{C}^{1,1}$ estimates are then used in two ways: 
\begin{itemize}
    \item[$\bullet$] First, they allow us to solve $(E_\varepsilon)$ for any fixed $\varepsilon \in (0,1]$ via the continuity method in $(\mathcal{C}^{k,\alpha})_0$ for any $k \geq 2$. Recall this is done by considering the family of equations $(E_s)$ with $s \in [\varepsilon,1]$, where openness in $(\mathcal{C}^{k,\alpha})_0$ follows from Proposition \ref{isolinMA}. The uniform $\mathcal{C}^{1,1}$ estimates that we will prove, together with the general regularity theory of the Monge-Amp\`ere equation, then imply closedness. Here, it is not yet important that the $\mathcal{C}^{1,1}$ estimates are uniform in $\varepsilon$, and the higher $\mathcal{C}^{k,\alpha}$ estimates will depend on $\varepsilon$ because the ellipticity of the equation does.  Also, note that these higher-order estimates follow from standard local regularity in the interior and from \cite[Section 2.1--2.2]{caffarelli1985dirichlet} near the boundary because we already have a true $\mathcal{C}^{1,1}$ bound.
    
    \item[$\bullet$] Once $(E_\varepsilon)$ is actually solved, we can then let $\varepsilon$ go to zero and use the uniformity of the $\mathcal{C}^{1,1}$ estimates of the $\varepsilon$-geodesic solution $\widetilde\varphi$ to extract a subsequential limit $\varphi \in \mathcal{C}^{1,1}$ such that $\Phi = \Psi + \varphi$ solves the geodesic equation \eqref{geoequ2}, \eqref{geoequ2p}, \eqref{geoequ2b}.
\end{itemize}

We omit these standard arguments and instead focus on the proof of the uniform $\mathcal{C}^{1,1}$ a priori estimates of the $\varepsilon$-geodesic solution $\widetilde\varphi$.  For this we follow the outline of \cite{Boucksom} in the compact case. However, we provide all the necessary details that are required to generalize this theory to the ALE case.  In addition, we also make use of the recent advance \cite{chu2017regularity} to obtain a $\mathcal{C}^{1,1}$ estimate which is uniform in $\varepsilon$.

In this section, we only deal with the uniform $\mathcal{C}^0$ estimate. We begin with a standard comparison principle \cite[Proposition 3.1]{bedford1976dirchlet}.

\begin{lem} \label{intcompare}
    Let $D$ be a bounded connected domain in $\CC^n$ with smooth boundary and $u,v \in \C^2 (D)$, plurisubharmonic functions in $D$. If $u= v $ on $ \partial D$ and $ u \geq v $, then we have
    \begin{align*}
        \int_{\Omega} (dd^c u )^n \leq \int_{\Omega} (dd^c v)^n.
    \end{align*}
\end{lem}

Then we can prove the following maximum principle for Monge-Amp\`{e}re operators.

\begin{thm} \label{mp2}
    Let $\Theta$ be a fixed reference K\"ahler form and $\Omega$, the pull-back of a semipositive $(1,1)$-form in $X$. Assume that $u,v\in  \C^2(X \times \Sigma)$ are bounded functions with $\Omega+dd^c v$, $\Omega+ dd^c u \geq 0$. If for some positive constants $\lambda$, $\Lambda$, we have the following properties: 
    \begin{align} 
        &(\Omega + dd^c v)^{n+1} \leq (\Omega + dd^c u)^{n+1} &\quad \text{ in } X \times \Sigma, \label{mp2con1}\\
        &\lambda\Theta \leq \Omega + dd^c u \leq \Lambda \Theta &\quad \text{ in } X \times \Sigma, \label{mp2con2}\\
        & u \leq v & \quad \text{ on } X \times \partial \Sigma, \label{mp2con3}
    \end{align}
       then $u \leq v$ in  $ X \times \Sigma$.
\end{thm}

\begin{proof}
    Assume $u (z_0) > v (z_0)$ at some point $z_0 \in X \times \Sigma$. Let $2h = u(z_0) - v(z_0)$. Then, we can modify $u,v$ to be $\tilde{u},\tilde{v}$ as follows:
    \begin{align}\label{mp2constr1}
    \begin{split}
    &\tilde{v} = v + h, \\
    &\tilde{u} = u + \frac{h}{2} |\tau|^2.
    \end{split}
    \end{align}
    It can be checked that $\tilde{u}$, $\tilde{v}$ are bounded functions satisfying that $\tilde{u} \leq \tilde{v}$ on $ X \times \partial \Sigma$ and $\tilde{u} (z_0) \geq \tilde{v} (z_0) + h$. By Wu-Yau's generalized maximum principle, there exists a sequence $\{p_k\}$ in $X \times \Sigma$ such that
    \begin{equation*}
        \lim_{k \rightarrow \infty} (\tilde{u} - \tilde{v})(p_k) = \sup_{X \times \Sigma} (\tilde{u}- \tilde{v}) \geq h, \quad \limsup_{k\rightarrow \infty} dd^c (\tilde{u}- \tilde{v}) (p_k) \leq 0.
    \end{equation*}
    For a sufficiently small constant $\delta >0$, there exist a point $p \in X \times \Sigma$, $dd^c \tilde{u} ( p ) - dd^c \tilde{v} (p) \leq \delta \Theta$ and $\eta_0 = \tilde{u} (p) - \tilde{v} (p) \geq \sup_{X \times \Sigma} (\tilde{u} - \tilde{v})- \delta$. Fix a local holomorphic chart around $p$, $\{U, z^{i}: i= 1, \ldots, n+1\}$ with $z^{n+1} = \tau $. Without loss of generality, we assume $U$ contains the unit disk in $\CC^{n+1}$ and for any local vector field $V \in T^{1,0}U$,
    \begin{align*}
        C^{-1} |V|^2 \leq \Theta (V, \overline{V}) \leq C |V|^2,
    \end{align*}
    where the constant $C$ only depends on the geometry of $X$ and the reference metric $\Theta$. Let $\mathbf{e} = 2C \delta$ and $\eta = \eta_0 - \frac{C \delta}{2} $. To derive the contradiction, we construct the following  local functions in $U$,
    \begin{align}\label{mp2constr2}
    \begin{split}
        &\dbtilde{u} = \tilde{u} - \mathbf{e} |z|^2,\\
        &\dbtilde{v} = \tilde{v} + \eta.
    \end{split}
    \end{align}
    If we denote the unit ball contained in the coordinate chart of $U$ by $B_1(p)$, we have $\dbtilde{u}(p) - \dbtilde{v} (p) = \frac{C\delta}{2}>0$ and $\dbtilde{u} \leq \dbtilde{v} $ on $ \partial B_1(p) $. Consider the following subset of $B_1(p)$,
    \begin{align*}
        D = \{z\in B_1(p): \dbtilde{u}(z) >\dbtilde{v} (z)\}.
    \end{align*}
    Let $\rho$ be the local potential of $\Omega$ in $U$, $\Omega= dd^c \rho$. According to Lemma \ref{intcompare}, 
    \begin{align}
        \int_D \big[dd^c (\rho + \dbtilde{v})\big]^{n+1} \geq \int_D \big[dd^c(\rho + \dbtilde{u})\big]^{n+1}.
    \end{align}
    Taking $\mathbf{e}\leq \frac{\lambda}{4C}$, 
    \begin{align*}
        dd^c(\rho + u - \mathbf{e} |z|^2) \geq \frac{1}{2} dd^c (\rho + u).
    \end{align*}
    Together with the construction of $\dbtilde{u}$ and $\dbtilde{v}$ in \eqref{mp2constr1}, \eqref{mp2constr2},
    \begin{align}
        \int_D \big[dd^c (\rho + v) \big]^{n+1} &\geq \int_D \Big[dd^c \big(\rho +u + \frac{h}{2} |\tau|^2 - \mathbf{e}|z|^2\big)\Big]^{n+1} \nonumber\\
        & \geq \frac{h \lambda^n}{2^{n+1}} \int_D \Theta^{n+1} + \int_D \big[dd^c (\rho +u)\big]^{n+1} - 2\mathbf{e} \Lambda^{n} \int_{D} \Theta^{n+1}. \label{mp2keycalc}
    \end{align}
    By picking $\mathbf{e}$ smaller, $\mathbf{e} \leq \frac{h \lambda^n}{2^{n+4} \Lambda^{n}}$, and combining with \eqref{mp2con1}, we have,
    \begin{align}
        \int_D \big[dd^c (\rho + v)\big]^{n+1} &\geq \int_D \big[dd^c (\rho +u)\big]^{n+1} + \frac{h \lambda^n}{2^{n+4}} \int_D \Theta^{n+1} \nonumber \\ &\geq \int_D \big[dd^c (\rho + v)\big]^{n+1} + \frac{h \lambda^n}{2^{n+4}} \int_D \Theta^{n+1}. \label{mp2keycontra}
    \end{align}
    Since the second term of \eqref{mp2keycontra} is strictly positive, which leads to a contradiction, we complete the proof.
\end{proof}

Let $\widetilde\varphi = \widetilde{\Phi}-\Psi$ be the solution of $(E_\varepsilon)$ after subtracting $\Psi$. According to Theorem \ref{mp2}, we have a uniform lower bound $\widetilde\varphi \geq 0$; hence, $\widetilde{\Phi} \geq \Psi$. The upper bound is easy to construct. Consider the  function defined in $X\times \Sigma$, $H = 2t(1-t)$. By restricting to each section $\Sigma_{x_0}= \{x_0\} \times \Sigma \stackrel{i_{x_0}}{\hookrightarrow} X \times \Sigma$, we have
\begin{align*}
   i_{x_0}^{*} ( \Theta_\Psi + dd^c H ) \leq 0 < i_{x_0}^* (\Theta_\Psi + dd^c \varphi). 
\end{align*}
Hence, $\Delta_\Sigma H \leq dd^c \widetilde\varphi $ in $\Sigma_{x_0}$ and $H= \varphi =0 $ on its boundary $\partial \Sigma_{x_0}$. The maximum principle on compact manifolds with boundary implies that $\widetilde\varphi \leq H$ on each section. Hence, we get the desired uniform $\C^0$ estimate,
\begin{align*}
    \Psi \leq \widetilde{\Phi} \leq \Psi + H.
\end{align*}
 
\section{A priori estimate up to $\C^1$}

For the $\C^1$ bound, Blocki gives an explicit estimate in the compact setting in \cite{blocki2009gradient}. We generalize this estimate to the non-compact case. The $\C^1$ boundary estimate follows directly from the fact that $\Psi \leq \widetilde{\Phi} \leq \Psi + H $ in $X \times \Sigma$ and $\Psi$, $\widetilde{\Phi}$, $\Psi + H$ agree along $X \times \partial \Sigma$. Let $\nabla$ be the Levi-Civita connection of $\Theta_\Psi$ on $X \times \Sigma$. Then we have
\begin{align*}
    |\nabla \widetilde{\Phi} |_{\Theta_\Psi} \leq \max \{ |\nabla \Psi|_{\Theta_\Psi}, |\nabla (\Psi + H)|_{\Theta_\Psi} \}, \quad \text{ on } X \times \partial \Sigma.
\end{align*}
Hence, $\sup_{X \times \partial \Sigma} |\nabla \widetilde{\Phi}|_{\Theta_\Psi} \leq C$, where $C$ is a uniform constant.

\begin{pro} \label{c1priest} Let $\widetilde\varphi = \widetilde{\Phi} -\Psi \in \C^{3}_{loc} (X \times \Sigma)$ be a solution of $(E_{\varepsilon})$ and let  $\nabla$ be the Levi-Civita connection of the Kähler metric $\Theta_\Psi$ on $X \times \Sigma$. Assume that $\widetilde\varphi$ lies in the space $\C^1(X \times \Sigma, \Theta_\Psi)$. Then, 
\begin{align*}
    \sup_{X \times \Sigma}|\nabla \widetilde\varphi|_{\Theta_\Psi} \leq C,
\end{align*}
where $C$ is a positive constant depending only on upper bounds for $|\widetilde\varphi|$, on lower bounds for the bisectional curvature of $\Theta_\Psi$, and on $n$, but not on $\varepsilon$.
\end{pro}

\begin{proof} Suppose that $\inf_{X \times \Sigma} \widetilde\varphi =A$ and $\sup_{X \times \Sigma} \widetilde\varphi =B$. Consider the following function,
\begin{align*}
    \alpha = \log \beta - \gamma \circ \widetilde\varphi,
\end{align*}
where $\beta = |\nabla \widetilde\varphi|^2_{\Theta_\Psi}$ and $\gamma : [A, B ] \rightarrow \RR$ is a smooth function to be determined later. According to the assumption that $\widetilde\varphi$ lies in the space $\C^1$,  Yau's maximum principle can be applied here. In particular, there exists a sequence in $\{x_k\}$ in $X \times \Sigma^{\circ}$ such that,
\begin{align*}
    \lim_{k \rightarrow \infty} \alpha (x_k) = \sup_{X \times \Sigma} \alpha, \quad \lim_{k \rightarrow \infty} |\nabla \alpha (x_k)|_{\Theta_\Psi} =0, \quad \limsup_{k \rightarrow \infty} \Delta \alpha (x_k) \leq 0,
\end{align*}
where $\Delta = \Delta_{\Theta_\Psi}$. Then, for a sufficiently small $\mathbf{e} >0$ to be determined later and all $k \gg 1$, we have
\begin{align} \label{devcon1}
   \alpha(x_k) \geq \sup_{X \times \Sigma} \alpha -\mathbf{e}, \quad |\nabla \alpha(x_k)|_{\Theta_\Psi} \leq \mathbf{e}, \qquad \Delta \alpha (x_k) \leq \mathbf{e}.
\end{align}
Fixing $O =x_k$ satisfying \eqref{devcon1}, we can pick the normal coordinates around $O$. Let $g$ and $\widetilde{g}$ denote the metric tensors corresponding to $\Theta_\Psi$ and $\Theta_{\widetilde{\Phi}} = \Theta_\Psi + dd^c \widetilde\varphi$. Then there exist local holomorphic coordinates near $O$ such that,
\begin{align*}
    g_{i\overline{j}} (O) = \delta_{ij}, \quad g_{i\overline{j}, k}(O) = 0 \quad \text{ and } \quad  \widetilde{g}_{i\overline{j}}(O) \text{ is diagonal}.
\end{align*}
By taking derivative of $\alpha $, 
\begin{align*}
    \alpha_p = \frac{\beta_p}{ \beta} - (\gamma' \circ \widetilde\varphi) \cdot \widetilde\varphi_p.
\end{align*}
Combining with condition \eqref{devcon1}, $|\alpha_p (O)| \leq \mathbf{e}$. Then, at the point $O$, we have
\begin{align} \label{2nddevtest1}
    \alpha_{p\overline{p}} \geq \frac{\beta_{p\overline{p}}}{\beta} - [(\gamma')^2 + \gamma''] |\widetilde\varphi_p|^2 -\gamma' \widetilde\varphi_{p\overline{p}} - \mathbf{e} |\gamma'| |\widetilde\varphi_p| - \mathbf{e}.
\end{align}
If we write the local potential of $\widetilde{g}_{i\overline{j}} $ as $u$ near $O$, then the $\varepsilon$-geodesic equation is locally given by $\det(u_{i\overline{j}}) =  \upsilon (\varepsilon) \det( g_{i\overline{j}})$. The direct derivative of the equation at $O$ gives,
\begin{align} \label{devcon2}
    \sum_{p}\frac{u_{p\overline{p} j}}{ u_{p\overline{p}}} = \big(\log \upsilon (\varepsilon) \big)_j.
\end{align}
Also, notice that,
\begin{align*}
    \beta_{p\overline{p}} \geq -D \beta + 2 \re \sum_{j} u_{p\overline{p}j}\widetilde\varphi_{\overline{j}} + \sum_j |\widetilde\varphi_{jp}|^2 + \widetilde\varphi_{p\overline{p}}^2,
\end{align*}
where $-D$ is the negative lower bound of bisectional curvature of $\Theta_\Psi$. Recall that we have the assumption $C^{-1}g_{i\overline{j}} \leq u_{i\overline{j}} \leq C g_{i\overline{j}}$ and $|\widetilde\varphi_p| < C$, where $C$ is the constant from our assumption at the beginning of this section and we will get rid of this constant in the end. Together with  \eqref{2nddevtest1} and \eqref{devcon2}, we have,
\begin{align} \label{c1key1}
\begin{split}
C \mathbf{e} \geq \sum_{p} \frac{\alpha_{p\overline{p}}}{u_{p\overline{p}}} \geq (\gamma' &-D) \sum_{p} \frac{1}{u_{p\overline{p}}} + \frac{1}{\beta} \sum_{jp}  
\frac{|\widetilde\varphi_{jp}|}{ u_{p\overline{p}}}\\ & -2 \re \frac{1}{\beta} \sum_{j} \big(\log \upsilon (\varepsilon) \big)_j \widetilde{\varphi}_{\overline{j}} \\ &- [(\gamma')^2+ \gamma''] \sum_{p} \frac{|\widetilde\varphi_p|^2}{ u_{p\overline{p}}} - n \gamma' - C(|\gamma'| +1)\mathbf{e}.
\end{split}
\end{align}
According to Blocki's key observation in \cite{blocki2009gradient}, after modified in our case, at the point $O$, we have
\begin{align*}
    \frac{1}{\beta} \sum_{j,p} \frac{|\widetilde\varphi_{jp}|^2}{u_{p\overline{p}}} \geq (\gamma')^2 \sum_p \frac{|\widetilde\varphi_p|^2}{u_{p\overline{p}}}- 2 \gamma' - \frac{2+ C\mathbf{e}}{\beta} - C (1+ |\gamma'|) \mathbf{e},
\end{align*}
and
assuming that $\beta \geq 1$, we have
\begin{align*}
    \frac{2}{\beta} \re \sum_j \big(\log \upsilon \big)_j \varphi_{\overline{j}} \geq -2 \frac{|\nabla \log \upsilon (\varepsilon)|}{\sqrt{\beta}} \geq - 2(n+1) \frac{\big|\nabla \big(\upsilon (\varepsilon)^{\frac{1}{n+1}}\big)\big|}{\upsilon(\varepsilon)^{\frac{1}{n+1}}} \geq - V \sum_{p} \frac{1}{u_{p\overline{p}}}
\end{align*}
where $V$ is a uniform constant satisfying  
\begin{align*}
    V \geq 2(n+1) \big| \nabla \big( \upsilon (\varepsilon)^{\frac{1}{n+1}}\big) \big|.
\end{align*}
Combining with \eqref{c1key1},
\begin{align} \label{c1key2}
    C(1+ |\gamma'|) \mathbf{e} \geq (\gamma' -D -V) \sum_{p} \frac{1}{u_{p\overline{p}}} - \gamma'' \sum_{p} \frac{|\widetilde\varphi_p|^2}{u_{p\overline{p}}}- (n+2) \gamma' -2.
\end{align}
Now, we choose the function $\gamma$ and the small number $\mathbf{e}>0$ in \eqref{c1key2} as follows. Let $\gamma = (D +V+3)(t-A) - (B-A)^{-1} (t-A)^2$ and $\mathbf{e} \leq C^{-1} (D+V+3)^{-1}$, then we have
\begin{align*}
    \sum_{p} \frac{1}{u_{p\overline{p}}} + \frac{2}{B-A} \sum_{p} \frac{|\widetilde\varphi_p|^2}{ u_{p\overline{p}}} \leq 3 + (n+2) (D+V+3).
\end{align*}
Then, it is straightforward to conclude that
$\beta (O) \leq \max \{ [(n+3)(D+V+3)]^{n+1} n(B-A) ,1\}$. Noting that $\beta \leq \exp \{ \mathbf{e} + \log \beta(O) - \gamma\circ\widetilde\varphi (O) + \gamma\circ \widetilde\varphi\}$, hence, $\beta$ is controlled by some uniform constant only depending on $\|\widetilde\varphi\|_{L^\infty}$, $D$, $V$ and $n$.
\end{proof}
\section{A priori estimate up to $\C^{1,1}$}\label{sectionC11}
First, we deal with the uniform $\C^2$ boundary estimate on $X \times \partial \Sigma$. The technique is to construct local barrier functions near the boundary, which is completely parallel to \cite{caffarelli1985dirichlet, chen2000space, guan1998dirichlet}. The statement is the following:
\begin{lem}\label{c2pribdyest}
Let the data $(X \times \Sigma,  \Theta_\Psi, \widetilde\varphi)$ be the same as in Proposition \ref{c1priest}. Let $\nabla$ denote the Levi-Civita connection of $\Theta_\Psi$ on $X \times \Sigma$. Then 
\begin{align*}
    \sup_{X\times \partial \Sigma}|\nabla^2 \widetilde\varphi|_{\Theta_\Psi} \leq C,
\end{align*}
where the constant $C$ only depends on $\sup_{X\times \Sigma} |\nabla \widetilde\varphi|_{\Theta_\Psi}$ and on $ (X \times \Sigma, \Theta_\Psi) $.
\end{lem}
\begin{proof} Fixing a point $p \in X \times \partial \Sigma$, we pick the local holomorphic coordinates around the point $p$ such that the coordinates system is normal in $X$ and $\Sigma$ direction, we still pick the standard coordinate function of the annulus, denoted by $\{x_1, \ldots, x_{2n}, x_{2n+1} =t, x_{2n} =s\}$ and the corresponding holomorphic coordinates, $z_{i} = x_{2i-1} + i x_{2i}$. Throughout the proof, we assume the metric tensor $g$ associated with $\Theta_\Psi$ satisfies $m \delta_{ij} \leq g_{i\overline{j}} \leq M \delta_{ij}$. In general, we need to prove the boundary $C^2$ estimate at $p$ in tangential-tangential, tangential-normal and normal-normal directions respectively. However, the tangential-tangential is trivial in our case and the normal-normal estimate follows directly from the tangential-normal estimate. Here, we briefly summarize the proof of tangential-normal estimate by explicitly constructing the barrier functions. 

Consider a small neighborhood near $p$, $ \displaystyle B'_{\delta}(p) = (X \times \Sigma) \cap  B_\delta (p)$, where the small constant $\delta$ will be determined later. Firstly, we construct the following auxiliary function in $B'_\delta (p)$,
\begin{align}\label{auxfun}
    v = \widetilde\varphi + N t(1-t),
\end{align}
where $N$ is a large constant to be determined. Then, it can be easily checked that 
\begin{align*}
    \widetilde{\Delta} v \leq n+1 - m \sum_{i}\widetilde{g}^{i\overline{i}} -  N \widetilde{g}^{n+1,\overline{n+1}},
\end{align*}
where $\widetilde{g}$ again denotes the metric tensor associated with $\Theta_{\widetilde{\Phi}} = \Theta_\Psi + dd^c \widetilde\varphi$ and $\widetilde{\Delta}$ denotes the corresponding Laplacian. Notice that
\begin{align*}
    -\frac{m}{2} \sum_i \widetilde{g}^{i\overline{i}} -N \widetilde{g}^{n+1, \overline{n+1}} \leq -  \frac{mN^{\frac{1}{n+1}}}{2(\det \widetilde{g})^{\frac{1}{n+1}}} = -\frac{m N^{\frac{1}{n+1}}}{ 2 \upsilon(\varepsilon)^{\frac{1}{n+1}}} (\det g)^{-\frac{1}{n+1}}.
\end{align*}
By taking $N = [(n+1) (2/m )]^{n+1} \max_{B_\delta'(p)} (\det g)$, we have $\displaystyle \widetilde{\Delta} v \leq -\frac{m}{2} \sum_{i} \widetilde{g}^{i\overline{i}}$. Noting that $\widetilde{\varphi} = \widetilde{\Phi} -\Psi \geq 0$, we have $v \geq 0$ on $\partial B'_\delta(p)$. Then, the barrier functions can be constructed as follows:
\begin{align*}
    w = A v + B |z|^{2} \pm \frac{\partial}{ \partial x_k} \widetilde\varphi, \quad \text{  for } 1\leq k \leq 2n \text{ or } k =2n+2.
\end{align*}
By differentiating the Monge-Amp\`{e}re equation $(E_\varepsilon)$ in the local coordinates, 
\begin{align*}
    \pm\widetilde{\Delta} \Big( \frac{\partial}{\partial x_k} \widetilde\varphi \Big) =\pm\big( \widetilde{g}^{i\overline{j}} (\widetilde{g})_{i\overline{j}, k} - \widetilde{g}^{i\overline{j}} g_{i\overline{j}, k} \big) \leq C (1 + \sum \widetilde{g}^{i\overline{i}}),
\end{align*}
where $A$ and $B$ are large positive constants to be determined. According to the $\C^1$ estimate of $\widetilde{\varphi}$, we assume that $|\partial_k \widetilde{\varphi}| \leq C$. By picking a very large constant $B$ such that, on $\partial B'_\delta (p)$, $B |z|^2 \pm \partial_k \widetilde{\varphi} \geq 0$, we have $w \geq 0$ on $\partial B'_\delta(p)$. Then, we choose a large constant $A$ such that $\widetilde{\Delta} w \leq 0$ in $B'_\delta (p)$. Then, by maximum principle, $w\geq 0$ in $B'_\delta (p)$. Together with the fact that $w(p) =0$, we have $ \partial_t w \geq 0$ at $p$, which implies the tangential-normal estimate on the boundary.
\end{proof}

Lemma \ref{c2pribdyest} together with Yau's standard calculation on Laplacian estimate implies the following interior Laplacian estimate, referring to \cite{Yau78}.

\begin{lem} \label{Yaulaest}
Let $\widetilde{\varphi}$ be the solution of $(E_{\varepsilon})$ and $\Delta$, $\widetilde{\Delta}$, the Laplacian operators of $g = \Theta_\Psi$ and $\widetilde{g} = \Theta_{\widetilde\Phi} = \Theta_\Psi + dd^c\widetilde\varphi$ respectively. Then, for any constant $C$,
\begin{align*}
    \widetilde{\Delta} \big(e^{-C \widetilde\varphi} (n+1 + \Delta \widetilde\varphi) \big) \geq & \quad e^{-C \widetilde\varphi}  \big( \Delta \log \upsilon(\varepsilon) - (n+1)^2 \inf_{i\ne l} (R_{i\overline{i}  l \overline{l}}) \big)\\
    & - C e^{-C \widetilde\varphi} (n+1) (n+1+ \Delta \widetilde\varphi)\\ & + (C + \inf_{i\ne l} (R_{i\overline{i}l \overline{l}})) e^{-C\widetilde\varphi} (n+1 + \Delta \widetilde\varphi)^{1+\frac{1}{n}} \upsilon( \varepsilon)^{-1},
\end{align*}
where $R$ denotes the curvature tensor of $g$. From this, we can deduce the estimate
\begin{align*}
    \sup_{X \times \Sigma}|\Delta \widetilde\varphi| \leq C (1+ \sup_{X \times \partial \Sigma} |\Delta \widetilde\varphi|),
\end{align*}
where $C$ only depends on $\sup_{X \times \Sigma} \widetilde\varphi$ and on a negative lower bound of $\inf_{i\ne l}(R_{i\overline{i}l \overline{l}})$.
\end{lem}

Lemma \ref{Yaulaest}, together with Lemma \ref{c2pribdyest}, implies that there exists a uniform constant $C$ only depending on $\sup_{X \times \Sigma} \Delta \widetilde\varphi$ such that $ \varepsilon C^{-1} g_{i\overline{j}} \leq  \widetilde{g}_{i\overline{j}} \leq C g_{i\overline{j}}$. This is already enough to apply the standard local regularity theory of the Monge-Amp\`ere equation to prove $\mathcal{C}^{k,\alpha}$ estimates for any $k \geq 2$ that depend on a positive lower bound for $\varepsilon$. In this way the equation $(E_\varepsilon)$ can be solved using the continuity path $(E_s)$, $s \in [\varepsilon,1]$. However, in order to construct an honest geodesic by letting $\varepsilon \to 0$, we require a full $\C^{1,1}$ estimate which is uniform in $\varepsilon$. In \cite{chu2017regularity}, $\C^{1,1}$ regularity is proved in the compact case. The method can also be applied in the ALE  K\"ahler setting. 

\begin{pro}
Let the data $(X \times \Sigma,  \Theta_\Psi, \widetilde\varphi)$ be the same as in Proposition \ref{c1priest}. If $\widetilde\varphi$ lies in the space $\C^2(X\times \Sigma,\Theta_\Psi)$, then there exists a constant $C$ such that
\begin{align*}
    |\nabla^2 \widetilde\varphi|_{\Theta_\Psi} \leq C,
\end{align*}
where $\nabla$ again denotes the Levi-Civita connection of the metric $\Theta_\Psi$ and $C$ depends only on $(X \times \Sigma, \Theta_\Psi)$ and on $\sup_{X\times \Sigma}|\widetilde\varphi|$, $\sup_{X \times \Sigma} |\nabla \widetilde\varphi|_{\Theta_\Psi}$, $\sup_{X \times \Sigma} |\Delta \widetilde\varphi|$, $\sup_{X \times \partial \Sigma} |\nabla^2 \widetilde\varphi|_{\Theta_\Psi}$.
\end{pro}

\begin{proof}
We again write $g$ for the metric tensor associated with $\Theta_\Psi$. Let $\lambda_1(\nabla^2 \widetilde\varphi)$ be the largest eigenvalue of the real Hessian $\nabla^2 \widetilde\varphi$. By observing that there exists a uniform constant $C$ such that $\lambda_1 (\nabla^2 \widetilde\varphi) \leq |\nabla^2 \widetilde\varphi|_g \leq C \lambda_1 (\nabla^2 \widetilde\varphi) +C $, it suffices to prove that $\lambda_1 (\nabla^2 \widetilde\varphi)$ has a uniform upper bound. Consider the following quantity,
\begin{align*}
    Q = \log \lambda_1 (\nabla^2 \widetilde\varphi) + h (|\nabla \widetilde\varphi|_g^2) -A \widetilde\varphi,
\end{align*}
where $h$ is defined to be $\displaystyle h(s) = -\frac{1}{2} \log \big (1 + \sup_{X \times \Sigma} |\nabla \widetilde\varphi|_g^2 -s  \big)$ and $A$ is a uniform large positive constant to be determined later. We can further modify this quantity to $Q_{\mathbf{e}} = Q - \mathbf{e} r$, where $\mathbf{e}$ is a small positive constant to be determined later.  According to the assumption that $|\nabla^2 \widetilde\varphi|$ is bounded and hence so is $Q$, the modified quantity $Q_\mathbf{e}$ attains its maximum at some point $x_\mathbf{e} \in X \times \Sigma$. The same argument as in Lemma \ref{openmax1} implies that $\lim_{\mathbf{e} \rightarrow 0} Q(x_\mathbf{e}) = \sup_{X \times \Sigma} Q $. In the following, we assume $\mathbf{e}$ is small enough such that $|Q(x_\mathbf{e})- \sup_{X \times \Sigma} Q| < 1 $ and always write $p=x_\mathbf{e}$. Since $Q_\mathbf{e}$ might not be smooth at $p$ if the eigenspace of $\lambda_1(\nabla^2 \widetilde\varphi) (p)$ has dimension greater than one, a perturbation argument used in \cite{chu2017regularity} can be applied to the quantity $Q_\mathbf{e}$ here. 

Fix normal coordinates $(z_1, \ldots, z_{n+1})$ with respect to $g$ at $p$ such that $(\widetilde\varphi_{i\overline{j}})$ is diagonal at $p$. Define the corresponding real coordinates $(x_1, \ldots, x_{2n})$ by $z_i = x_{2i-1} + i x_{2i}$. Let $\lambda_1 \geq \lambda_2 \geq \ldots \geq \lambda_{2n}$ be the eigenvalues of $\nabla^2 \widetilde\varphi$ at $p$ and $V_1, \ldots,V_{2n}$, the corresponding unit eigenvectors at $p$. The eigenvectors can be extended to vector fields with constant coefficients in a small neighborhood of $p$, also denoted by $V_1, \ldots, V_{2n}$, and can be represented by $V_\alpha = V^\beta_{\alpha} {\partial_{x_\beta}}$ in the local coordinates. The perturbation argument is to perturb $\nabla^2 \widetilde\varphi$ locally around $p$ and to ensure that $\lambda_1 > \lambda_2$ near $p$. Precisely, consider the following locally defined tensor field, 
\begin{align*}
    P = \sum_{\alpha, \beta} \big(\delta_{\alpha\beta} - V^{\alpha}_1 V^{\beta}_1 \big) dx_\alpha \otimes dx_\beta.
\end{align*}
Let $\lambda_i' = \lambda_i (\nabla^2 \widetilde\varphi - P)$. Then, one can easily check that $\lambda'_1 (p)= \lambda_1(p)$ and $\lambda'_i (p)= \lambda_i (p)-1$ for $i \geq 2$. Hence, there exists a neighborhood of $p$ such that $\lambda'_1 > \lambda'_2 \geq \ldots \geq \lambda'_{2n}$ and $\lambda'_1 \leq \lambda_1 $. Consider the following perturbed quantities,
\begin{align*}
    \hat{Q} = \log \lambda'_1 + h(|\nabla \widetilde\varphi|^2_g) -A \widetilde\varphi, \quad \hat{Q}_{\mathbf{e}} = \hat{Q} - \mathbf{e} r.
\end{align*}
Therefore, $ \hat{Q}_\mathbf{e}$ is a smooth quantity with a local maximum at $p$. Then, we have,
\begin{align*}
  |d \hat{Q}(p)|_g \leq C\mathbf{e}, \quad \Delta \hat{Q} (p) \leq C\mathbf{e}.
\end{align*}

The following inequality follows directly from the calculation in \cite[Lemma 2.1]{chu2017regularity}. The only information we need in the calculation is the second derivative of the Monge-Amp\`{e}re equation at $p$. We will not repeat the details here.
By assuming $\lambda_1' \geq 1$ at $p$, and again writing $\widetilde{g}$ for the metric tensor associated with $\Theta_{\widetilde{\Phi}}=\Theta_\Psi+dd^c\widetilde\varphi$, we have
\begin{align} \label{laQ1}
\begin{split}
    \Delta \hat{Q} \geq & \ 2 \sum_{\alpha> 1} \frac{\widetilde{g}^{i\overline{i}}|\partial_i (\widetilde\varphi_{V_\alpha V_1})|^2}{\lambda_1 (\lambda_1- \lambda_\alpha)} + \frac{\widetilde{g}^{i\overline{i} }\widetilde{g}^{j\overline{j}} \big|V_1 \big( \widetilde{g}_{i\overline{j}}\big)\big|^2}{\lambda_1} - \frac{\widetilde{g}^{i\overline{i}} |\partial_i (\widetilde\varphi_{V_1 V_1})|^2}{\lambda_1^2} \\
    & + h' \sum_{k} \widetilde{g}^{i\overline{i}} \big( |\widetilde\varphi_{ik}|^2 + |\widetilde\varphi_{i\overline{k}}|^2 \big) + h'' \widetilde{g}^{i\overline{i}} \big|\partial_i |\nabla \widetilde\varphi|^2_g \big|\\
    & + (A-B) \sum_{i} \widetilde{g}^{i\overline{i}} -A n,
\end{split}
\end{align}
where the constant $B$ only depends on $(X\times \Sigma, g)$ and $\sup_{X \times \Sigma} |\nabla \widetilde\varphi|_g$. To cancel the annoying terms, we deal with the third term in \eqref{laQ1}, $\displaystyle {{\lambda_1^{-2}}\widetilde{g}^{i\overline{i}} |\partial_i (\widetilde\varphi_{V_1 V_1})|^2}$. To estimate the term, we split it into the following two parts,
\begin{align*}
  & I_1 = (1-2\delta) \frac{\widetilde{g}^{i\overline{i}} |\partial_i (\widetilde\varphi_{V_1 V_1})|^2}{\lambda_1^2}, \\& I_2 = 2\delta  \frac{\widetilde{g}^{i\overline{i}} |\partial_i (\widetilde\varphi_{V_1 V_1})|^2}{\lambda_1^2},
\end{align*}
where $0< \delta <1/4$ is to be determined later. For $I_1$, referring to \cite[Lemma 2.2]{chu2017regularity}, by assuming that $\lambda_1' \geq D /\delta$, where $D$ only depends on $(X\times \Sigma, g) $ and $\sup_{X \times \Sigma} \Delta \widetilde\varphi$, we have 
\begin{align} \label{I1est}
    I_1 \leq \sum_{i,j} \frac{\widetilde{g}^{i\overline{i}} \widetilde{g}^{j\overline{j}} \big|V_1 \big( \widetilde{g}_{i\overline{j}} \big)\big|}{\lambda_1} + 2 \sum_{\alpha >1} \sum_{i} \frac{\widetilde{g}^{i\overline{i}} |\partial_i (\widetilde\varphi_{V_\alpha V_1})|^2}{\lambda_1(\lambda_1 -\lambda_\alpha)} + \sum_{i} \widetilde{g}^{i\overline{i}}.
\end{align}
To estimate $I_2$, recall the fact that $d \hat{Q}_\mathbf{e} = 0$ and apply the derivative of eigenvalues referring to \cite[Lemma 5.2]{chu2017regularity}. Then, we have
\begin{align} \label{I2est}
\begin{split}
    I_2 &= 2\delta \sum_{i} \widetilde{g}^{i\overline{i}} \big|A \widetilde\varphi_i + h' \partial_i |\nabla \widetilde\varphi|^2_g - \mathbf{e} r_i\big|^2 \\
    & \leq 8 \delta A^2 \sum_{i } \widetilde{g}^{i\overline{i}} |\widetilde\varphi_i|^2 + 2 (h')^2 \sum_{i} \widetilde{g}^{i\overline{i}} |\partial_i |\nabla \widetilde\varphi|^2_g|^2 + C\mathbf{e}\sum_i \widetilde{g}^{i\overline{i}}.
\end{split}
\end{align}

Combining \eqref{laQ1}, \eqref{I1est}, \eqref{I2est} and $\Delta \hat{Q} \leq C\mathbf{e}$, then, by assuming $\lambda_1'  \geq D/\delta $, we have 
\begin{align*}
    C \mathbf{e} \geq &\ h' \sum_k \widetilde{g}^{i\overline{i}} \big( |\widetilde\varphi_{ik}|^2 + |\widetilde\varphi_{i\overline{k}}|^2 \big) +\big( h''- 2 (h')^2 \big)  \widetilde{g}^{i\overline{i}} \big| \partial_i |\nabla \widetilde\varphi|^2 \big| \\
    & -8 \delta A^2 \widetilde{g}^{i \overline{i}} |\widetilde\varphi _i|^2  +(A-B -C \mathbf{e}) \sum_{i} \widetilde{g}^{i \overline{i}} -A n.
\end{align*}
Notice that $h'' = 2(h')^2$. Picking $\mathbf{e} \leq 1/C $, $A = B+2$ and $\displaystyle \delta = \big(8 A^2 (\sup_{X \times \Sigma} |\nabla \widetilde\varphi|^2+1)\big)^{-1}$, then we have
\begin{align*}
    h' \sum_k \widetilde{g}^{i\overline{i}} \big(|\widetilde\varphi_{ik}|^2 +|\widetilde\varphi_{i\overline{k}}|^2\big) + \sum_{i} \widetilde{g}^{i\overline{i}} \leq An+1.
\end{align*}
Recall $\widetilde{g}_{i\overline{j}} \leq C g_{i\overline{j}} $, where $ C $ only depends on $\sup_{X \times \Sigma} \Delta \widetilde\varphi$. Hence, at $p$, $\widetilde{g}^{i\overline{i}} \geq C^{-1}$. Then, 
\begin{align*}
    \lambda_1 (p) \leq \max \Big\{ \frac{D}{\delta}, \big\{(An +1)C -n \big\} (1+ \sup_{X \times \Sigma} |\nabla \widetilde\varphi|_g^2) \Big\}.
\end{align*}
Together with the fact that $\sup_{X \times \Sigma} Q \leq Q(p) +1 $, we prove that $\sup_{X \times \Sigma} \lambda_1 $ is bounded by some uniform constant. 
\end{proof}

\section{The asymptotic behavior of $\varepsilon$-geodesics} 
\label{secasymbehavior} 

In this section, we prove Theorem \ref{epsilongeodesicasymptotics} on the asymptotic behavior of $\varepsilon$-geodesics for a fixed $\varepsilon>0$.
We use the notation introduced before Theorem \ref{epsilongeodesicasymptotics} and we assume $\psi_0,\psi_1 \in \HH_{-\gamma} (\omega)$ $(\gamma >0)$.
We are really interested in the case when $-\gamma = 2 -2\tilde{\tau} $ due to theorem \ref{yao2022}. In $\varepsilon$-geodesic equation $(E_\varepsilon)$, the derivatives of function $\upsilon (\varepsilon)$ decays at infinity with order $-\varsigma$, $|\nabla^k \upsilon (\varepsilon)| \leq O(r^{-\varsigma-k})$ with $\varsigma \geq \gamma$ for $k \geq 1$.
Without loss of generality, we assume $\varsigma \geq \gamma >\tau$, otherwise theorem \ref{epsilongeodesicasymptotics} can be proved more easily without iteration (step 3).

We also write $\varphi_\varepsilon = \Phi_\varepsilon -\Psi$, so that the solution is given by $\Theta + dd^c\Phi_\varepsilon = \Theta_\Psi + dd^c\varphi_\varepsilon$ with $\varphi_\varepsilon = 0$ on $X \times \partial\Sigma$. In Aleyasin \cite{aleyasin2014space}, a rough idea is given to prove the asymptotic behavior of $\varepsilon$-geodesics by constructing barrier functions in the (strictly easier) special case where the asymptotic coordinates are $J$-holomorphic and the decay rate of the ALE Kähler metric to the Euclidean metric is high enough. However, even in this special case, the details are more involved than what is suggested in \cite{aleyasin2014space}. Here we give a complete proof in the general setting.

\subsubsection*{Step 1: Differentiating the Monge-Amp\`{e}re equation.}

The Monge-Amp\`{e}re equation can be written explicitly in the asymptotic coordinates of $ X \times \Sigma$. As the complex structure $J$ of $X$ does not coincide with the Euclidean complex structure $J_0$ of the asymptotic coordinates in general, we will use real coordinates for clarity. By passing to the universal covering of the end, we are able to work with the global coordinates. Precisely, let $\{z_1,\ldots, z_n\}$ be the asymptotic complex coordinates of $\CC^n \backslash B_R$ and $w = t+is$ the complex coordinate of $\Sigma$. The corresponding real coordinates are $\{x_{1}, \ldots, x_{2n},$ $x_{2n+1} = t, x_{2n+2}={s}\}$, where $z_k= x_{2k-1} + i x_{2k}$ for $k = 1,\ldots,n$. From now on:
\begin{itemize}
    \item[$\bullet$] Latin indices $i,j, \ldots$ will denote the real coordinates from $1$ to $2n+2$.
    \item[$\bullet$] Greek indices $ \alpha, \beta, \ldots$ will denote the real coordinates from $1$ to $2n$.
    \item[$\bullet$] The bold Greek indices $\boldsymbol{\mu},\boldsymbol{\nu}$ will denote the real coordinates from $2n+1$ to $2n+2$.
\end{itemize}
In these coordinates, we write the Riemannian metric tensors corresponding to $\Theta_\Psi$ and $\Theta_\Psi + dd^c \varphi_\varepsilon $ as $g_{ij}$ and $(g_{\varphi_\varepsilon})_{ij}$, respectively.

Throughout this section, we work in the asymptotic chart of $X$. This allows us to use the Euclidean metric on $(\RR^{2n} \setminus B_R) \times \Sigma$ as a reference metric to measure derivatives. This is helpful because it enables us to write down equations with a good structure. Let $|\cdot|_0$ denote the Euclidean length, $\nabla_0$ the Euclidean Levi-Civita connection and $\nabla_{0,X}$ ($\nabla_{0,\Sigma}$) the component of $\nabla_0$ acting only in the space (time) directions on $(\RR^{2n}\setminus B_R) \times \Sigma$.

Then, the equation $(E_\varepsilon)$ can be written as 
\begin{align} \label{maasym1}
    \sqrt{\det{\big( (g_{{\varphi_\varepsilon}})_{ij} \big)}} = \upsilon (\varepsilon) \sqrt{\det(g_{ij})}.
\end{align}
Recall that $\upsilon$ satisfies conditions in \eqref{conupsilon}.
By differentiating the log of both sides by $D_\alpha = \partial / \partial_{x_\alpha}$, we have
\begin{align}\label{maasym1der1}
    g_{\varphi_\varepsilon}^{ij} D_\alpha (g_{\varphi_\varepsilon})_{ij} = g^{ij} D_\alpha g_{ij} + D_\alpha \log \upsilon(\varepsilon).
\end{align}
The first goal is to rewrite the equation \eqref{maasym1der1} to be an elliptic equation in terms of $D_\alpha {\varphi_\varepsilon}$. Let $e_1, \ldots, e_{2n+2}$ represent the real coordinate vector fields of $x_1, \ldots, x_{2n+2}$. Notice that $ (g_{\varphi_\varepsilon})_{ij} = g_{ij} + dd^c {\varphi_\varepsilon} (e_i, J e_j) $. We compute $D_\alpha$ of the second term:
\begin{align} \label{calcrealcx}
\begin{split}
    D_{\alpha} [dd^c   {\varphi_\varepsilon} (e_i, Je_j )] = & -d \circ J \circ d ( D_\alpha {\varphi_\varepsilon}) (e_i , Je_j) - d \circ (D_{\alpha} J) \circ d {\varphi_\varepsilon} (e_i, J e_j) \\ & - d \circ J \circ d {\varphi_\varepsilon} (e_i, (D_{\alpha} J) e_j).
\end{split}
\end{align}
Observe that $D_\alpha J$ is completely horizontal because $J$ preserves the product structure of the tangent bundle $T((\RR^{2n} \setminus B_R) \times \Sigma)$ and $J|_{T\Sigma}$ is constant. Thus, 
\begin{equation}\label{eq:cancelJ}
    D_\alpha J = (D_\alpha J)^\beta_\xi (e_\xi^* \otimes e_\beta),\quad (D_\alpha J)_\xi^{\boldsymbol{\mu}} = 0, \quad (D_\alpha J)^\beta_{\boldsymbol{\nu}} = 0, \quad (D_\alpha J)^{\boldsymbol{\mu}}_{\boldsymbol{\nu}} = 0,
\end{equation}
where the coefficients $(D_\alpha J)^\beta_\xi$ depend only on $x_1,\ldots,x_{2n}$ and not on $x_{2n+1},x_{2n+2}$.
In the same way, we can also see that
\begin{equation}
|\nabla_{0,X}^m (D_\alpha J)|_0 = O(r^{-\tau-1-m}) \;\,(\text{all}\;m \geq 0),\;\,\nabla_{0,\Sigma}^m (D_\alpha J) = 0\;\,(\text{all}\;m \geq 1).
\end{equation}
Moreover, it is obvious that
\begin{equation}
    \Delta_{g_{\varphi_\varepsilon}} (D_\alpha  {\varphi_\varepsilon}) = \tr_{g_{{\varphi_\varepsilon}}} (dd^c (D_\alpha {\varphi_\varepsilon}) (\cdot, J\cdot)) = g^{ij}_{{\varphi_\varepsilon}} dd^c (D_\alpha {\varphi_\varepsilon}) (e_i, J e_j).
\end{equation}
Then, \eqref{calcrealcx}--\eqref{eq:cancelJ} imply that
\begin{align}
\begin{split}
    g_{\varphi_\varepsilon}^{i j} D_\alpha [dd^c {\varphi_\varepsilon} (e_i, J e_j)] &= \Delta_{g_{\varphi_\varepsilon}} (D_\alpha {\varphi_\varepsilon}) +  \mathbf{O}(r^{-\tau-1}) \circledast g_{\varphi_\varepsilon}^{-1} \circledast \nabla_0 \nabla_{0,X } {\varphi_\varepsilon}\\
    &+  \mathbf{O}(r^{-\tau-2}) \circledast g_{\varphi_\varepsilon}^{-1} \circledast \nabla_{0,X}\varphi_\varepsilon,
\end{split}
\end{align}
where $\circledast$ denotes a contraction and $\mathbf{O}$ denotes the following behavior of a tensor $T$:
\begin{equation*}
T = \mathbf{O}(r^{-\rho}) :\Longleftrightarrow |\nabla_{0,X}^mT|_0 = O(r^{-\rho-m})\;\,(\text{all}\;m \geq 0), \;\, \nabla_{0,\Sigma}^m T = 0\;\,(\text{all}\;m \geq 1).
\end{equation*}
Then, abbreviating the estimates
\begin{equation*}\begin{split}
    |\nabla_{0,X}^m(D_\alpha g_{ij})|_0 = &\;O(r^{-\tau-1-m})\;\,(\text{all}\;m \geq 0),\\
|\nabla_{0,X}^m\nabla_{0,\Sigma}(D_\alpha g_{ij})|_0 = O(r^{-2\tau-1-m})\;\,&(\text{all}\;m \geq 0),\;\,\nabla_{0,\Sigma}^m(D_\alpha g_{ij}) = 0\;\,(\text{all}\;m \geq 2),
\end{split}\end{equation*}
by $D_\alpha g_{ij} = \widehat{\mathbf{O}}(r^{-\tau-1})$ and 
\begin{align*}
    |\nabla^m_{0,X} \nabla^k_{0, \Sigma} D_\alpha \log \upsilon (\varepsilon)|_0 = O(r^{-\varsigma -1-m}) \ (\text{all } m, k \geq 0),
\end{align*}
by $D_\alpha \log \upsilon (\varepsilon) = \widehat{\widehat{\mathbf{O}}} (r^{ -\varsigma -1}) $
the equation \eqref{maasym1der1} can be rewritten as
\begin{align} \label{maasym1der2}
\begin{split}
    \Delta_{g_{\varphi_\varepsilon}} (D_\alpha {\varphi_\varepsilon}) &= \mathbf{O}(r^{-\tau-1}) \circledast g_{\varphi_\varepsilon}^{-1} \circledast \nabla_0 \nabla_{0,X } {\varphi_\varepsilon}\\
    &+ \mathbf{O}(r^{-\tau-2}) \circledast g_{\varphi_\varepsilon}^{-1} \circledast \nabla_{0,X}\varphi_\varepsilon\\
    &+ D_\alpha g_{ij} \cdot (g_{\varphi_\varepsilon}^{ij} - g^{ij}), \quad D_\alpha g_{ij} = \widehat{\mathbf{O}}(r^{-\tau-1}),\\
    &+ \widehat{\widehat{\mathbf{O}}} (r^{-\varsigma-1})
\end{split}
\end{align}

We will later use this formula in full but for now it is enough to take absolute values. Using the fact that $ \Lambda^{-1} \varepsilon \delta_{i j} \leq (g_{{\varphi_\varepsilon}})_{i j} \leq  \Lambda  \delta_{i j}$, and according to the uniform estimates of $|\nabla_0 {\varphi_\varepsilon}|_0$ and $|\nabla_0^2 {\varphi_\varepsilon}|_0$ from Theorem \ref{muniformc11}, the formula \eqref{maasym1der2} implies that  
\begin{align}\label{eq:PDEforD1}
    |\Delta_{g_{\varphi_\varepsilon}} (D_\alpha {\varphi_\varepsilon})| \leq C\varepsilon^{-1} r^{-\tau-1}
\end{align}  
for some constant $C = C(\|{\varphi_\varepsilon}\|_{\C^2(X \times \Sigma,\Theta_\Psi)}, \Lambda, g, J)$ bounded above independently of $\varepsilon$.

\subsubsection*{Step 2: Barrier estimate of the first derivatives.}

The next target is to construct the upper barrier and lower barrier functions to control $|D_\alpha {\varphi_\varepsilon}|$. Consider a smooth cutoff function $\chi: \RR_{\geq 0}\rightarrow \RR$ satisfying $\chi(x) = 0$ for $x\leq 1$, $\chi(x) =1$ for $x \geq 2$ and $|\chi'(x)| \leq 4$, $|\chi'' (x)| \leq 4$ for $1 \leq x \leq 2$. The function $D_\alpha {\varphi_\varepsilon}$ can be extended smoothly to $X \times \Sigma$ by defining
\begin{align} \label{extenc1}
    h = 
    \chi_{R_0} \cdot D_\alpha {\varphi_\varepsilon},
\end{align}
where $R_0$ is a large positive constant to be determined later such that $\{r(p) \geq R_0 /2\}$ is contained in the asymptotic chart of $X$ and $\chi_{R_0} (p) = \chi (r(p )/ R_0)$. From \eqref{eq:PDEforD1},
\begin{align} \label{weilaestc11} \begin{split}
    |\Delta_{g_{\varphi_\varepsilon}} h| \leq \begin{cases}
        C\varepsilon^{-1} r^{-\tau-1}, &\text{for}\;\,r \geq 2R_0,\\
    4 \Lambda \varepsilon^{-1} \big( R_0^{-2} |\nabla_{0,X} {\varphi_\varepsilon}|_0 +  R_0^{-1} |\nabla_{0,X}^2 {\varphi_\varepsilon}|_0 \big)  + C\varepsilon^{-1} R_0^{-\tau-1},  &\text{for}\;\,R_0 \leq r \leq 2 R_0,\\
   0, &\text{for}\;\,r \leq R_0.
    \end{cases}
\end{split}
\end{align}
Then, we can pick a barrier function as follows: 
\begin{align} \label{baweifunct1}
   u_1 = E \Big\{ \big(1-\chi_{\frac{R_0}{2}}\big) \Big(\frac{R_0}{2}\Big)^{-\tau-1} t(t-1) + \chi_{\frac{R_0}{2}} t(t-1) r^{-\tau-1} \Big\} \leq 0,
\end{align}
where the constant $E$ is to be determined later. The barrier function $u_1$ is defined in $ X\times \Sigma$ with $u_1=0$ on $X \times \partial \Sigma$. We also have
\begin{align*}
    \Delta_{g_{\varphi_\varepsilon}} u_1 &=\frac{1}{2}\tr_{g_{\varphi_\varepsilon}}(dd^c u_1(\cdot,J\cdot)) = \frac{1}{2} \left\{ \sum_{1 \leq \alpha,\beta \leq 2n} g^{\alpha\beta}_{{\varphi_\varepsilon}} (u_{1,\alpha\beta} + u_{1,J\alpha, J\beta} )\right.\\
    &\left.+ \sum_{\substack{1\leq \alpha \leq 2n,\\ 2n+1 \leq \boldsymbol{\mu} \leq 2n+2}}  g^{\boldsymbol{\mu} \alpha}_{{\varphi_\varepsilon}}(  u_{1,\alpha \boldsymbol{\mu}} + u_{1,J \alpha, J {\boldsymbol{\mu}}} )
   + \sum_{2n+1 \leq \boldsymbol{\mu},\boldsymbol{\nu} \leq 2n+2 } g_{\varphi_\varepsilon}^{\boldsymbol{\mu} \boldsymbol{\nu}} (u_{1,\boldsymbol{\mu}\boldsymbol{\nu}} + u_{1,J \boldsymbol{\nu},J\boldsymbol{\nu}})\right\}.
\end{align*}
Using the estimate $\Lambda^{-1} \varepsilon \delta_{i j} \leq (g_{{\varphi_\varepsilon}})_{i j} \leq  \Lambda  \delta_{i j}$, we obtain that
\begin{align} \label{laofba1}
    \Delta_{g_{\varphi_\varepsilon}} u_1 \geq \begin{cases}
        E ( \Lambda^{-1} r^{-\tau-1} - \Lambda \varepsilon^{-1} r^{-\tau-2} - \Lambda \varepsilon^{-1} r^{-\tau-3}) &\text{for}\;\,r \geq R_0/2,\\
        E  \Lambda^{-1} R_0^{-\tau-1} &\text{for}\;\,r \leq R_0/2.
    \end{cases}
\end{align}
By taking 
\begin{align}
   R_0 \geq  4 \Lambda^2 \varepsilon^{-1}, \quad E = 8 R_0^{\tau}C \;\,\text{with}\;\, C= C( \|{\varphi_\varepsilon}\|_{\C^2(X \times \Sigma,\Theta_\Psi)}, \Lambda, g, J),
\end{align}
and comparing with the inequality \eqref{weilaestc11}, we have $\Delta_{g_{\varphi_\varepsilon}} u_1 \geq \Delta_{g_{\varphi_\varepsilon}} h$. Together with the fact that $ u_1= h =0 $ on $X \times \partial\Sigma$, Lemma \ref{openmax1} implies that $h \geq u_1$ in $X \times \Sigma$. The same method shows the upper bound $h \leq -u_1 $, which, together with the lower bound, implies that for each spatial index $1\leq \alpha \leq 2n$,  
\begin{align} \label{maasym1der3}
    |D_\alpha {\varphi_\varepsilon}| \leq C\big( \|{\varphi_\varepsilon}\|_{C^2(X \times \Sigma,\Theta_\Psi)}, \Lambda, g, J, \varepsilon^{-1}\big)t(1-t) r^{-\tau-1}\;\,\text{on}\;\,\{r \geq 2R_0\} \times \Sigma.
\end{align}

\subsubsection*{Step 3: Barrier estimate of the second derivatives.}

Now, it comes to deal with the asymptotic behavior of the second derivative.

For a preliminary estimate, we go back to the full formula \eqref{maasym1der2} for $\Delta_{g_{\varphi_\varepsilon}}(D_\alpha \varphi_\varepsilon)$. For every $\mathbf{a} \in (0,1)$, the Euclidean $\mathcal{C}^{0,\mathbf{a}}$ norm of the right-hand side on a restricted unit ball $\hat{B}_1(p) = B_1(p) \cap ((\RR^{2n} \setminus B_{R}) \times \Sigma)$ with $r(p) = r \geq 2R$ is still bounded by $C( \|{\varphi_\varepsilon}\|_{C^2(X \times \Sigma,\Theta_\Psi)}, \Lambda, g, J, \varepsilon^{-1}) r^{-\tau-1}$ thanks to the Evans-Krylov estimates applied to $\varphi_\varepsilon$ in the interior and the estimates of \cite[Sections 2.1--2.2]{caffarelli1985dirichlet} at the boundary. (The precise dependence of this constant on the ellipticity, and hence on $\varepsilon^{-1}$, is not clear but also not needed.) Likewise, the $\C^{0,\mathbf{a}}$ norm of the coefficient tensor of the PDE, $g_{\varphi_\varepsilon}^{-1}$, is bounded by $C( \|{\varphi_\varepsilon}\|_{C^2(X \times \Sigma,\Theta_\Psi)}, \Lambda, g, J, \varepsilon^{-1})$. Applying the classic interior and boundary Schauder estimates to \eqref{maasym1der2}, we thus obtain from \eqref{maasym1der3} that
\begin{align} \label{maasymhigh1}
\|D_\alpha {\varphi_\varepsilon}\|_{\C^{2,\mathbf{a}}(\hat{B}_1(p))} \leq C\big(\|\varphi_\varepsilon\|_{C^2(X \times \Sigma, \Theta_\Psi)}, \Lambda, g, J, \varepsilon^{-1}\big) r^{-\tau-1}.
\end{align}

These estimates will now be used to start a bootstrap to obtain some decay for $D_\beta D_\alpha \varphi_\varepsilon$ using the same barrier method as in Step 2. Differentiate the equation \eqref{maasym1der2} again by $D_{\beta} = \partial /\partial x_\beta$ for $1\leq \beta \leq 2n$. This yields
\begin{align} \label{maasym2der1}
\begin{split}
   &\Delta_{g_{\varphi_\varepsilon}} (D_\beta D_\alpha {\varphi_\varepsilon}) = g_{\varphi_\varepsilon}^{-1} \circledast D_\beta g_{\varphi_\varepsilon} \circledast g_{\varphi_\varepsilon}^{-1} \circledast \nabla_0^2 \nabla_{0,X} \varphi_\varepsilon\\
   &+\mathbf{O}(r^{-\tau-2}) \circledast g_{\varphi_\varepsilon}^{-1} \circledast \nabla_0 \nabla_{0,X}\varphi_\varepsilon \\
   &+ \mathbf{O}(r^{-\tau-1}) \circledast g_{\varphi_\varepsilon}^{-1} \circledast D_\beta g_{\varphi_\varepsilon} \circledast g_{\varphi_\varepsilon}^{-1} \circledast \nabla_0 \nabla_{0,X}\varphi_\varepsilon \\
   &+ \mathbf{O}(r^{-\tau-1}) \circledast g_{\varphi_\varepsilon}^{-1} \circledast \nabla_0 \nabla_{0,X}^2\varphi_\varepsilon\\
   &+ \mathbf{O}(r^{-\tau-3}) \circledast g_{\varphi_\varepsilon}^{-1} \circledast \nabla_{0,X}\varphi_\varepsilon\\
   &+ \mathbf{O}(r^{-\tau-2}) \circledast g_{\varphi_\varepsilon}^{-1} \circledast D_\beta g_{\varphi_\varepsilon} \circledast g_{\varphi_\varepsilon}^{-1} \circledast \nabla_{0,X}\varphi_\varepsilon\\
   &+ \mathbf{O}(r^{-\tau-2}) \circledast g_{\varphi_\varepsilon}^{-1} \circledast \nabla_{0,X}^2\varphi_\varepsilon\\
   &+D_\beta D_\alpha g_{ij} \cdot (g_{\varphi_\varepsilon}^{ij} - g^{ij}), \quad D_\beta D_\alpha g_{ij} = \widehat{\mathbf{O}}(r^{-\tau-2}),\\
   &+\widehat{\mathbf{O}}(r^{-\tau-1}) \circledast  (g_{\varphi_\varepsilon}^{-1}  \circledast D_\beta g_{\varphi_{\varepsilon}} \circledast  g_{\varphi_\varepsilon}^{-1} - g^{-1} \circledast D_\beta g \circledast g^{-1})\\
   & + \widehat{\widehat{\mathbf{O}}} (r^{-\varsigma- 2})
   \end{split}
\end{align}
As before, we have that $\Lambda^{-1}g^{-1} \leq g_{\varphi_\varepsilon}^{-1} \leq \varepsilon^{-1} \Lambda g^{-1}$, and we also have
\begin{equation}
    |D_\beta g_{\varphi_\varepsilon}|_0 \leq |D_\beta g|_0 + |\nabla_{0,X} \nabla_0^2 \varphi_\varepsilon|_0 = O(r^{-\tau-1})
\end{equation}
thanks to the preliminary estimate \eqref{maasymhigh1}. Similarly, all derivatives of $\varphi_\varepsilon$ on the right-hand side of \eqref{maasym2der1} are at worst of order $3$, with at least one purely spatial derivative, and hence can be bounded by $O(r^{-\tau-1})$ thanks to \eqref{maasymhigh1}. In this way, we obtain that
\begin{align}\label{eq:goo}
\Delta_{g_{\varphi_\varepsilon}} (D_\beta D_\alpha {\varphi_\varepsilon}) = \widehat{\mathbf{O}}(r^{-\tau-2}) \circledast (g_{\varphi_\varepsilon}^{-1} - g^{-1}) + O(r^{-2\tau-2}).
\end{align}
The majority of terms on the right-hand side actually decay faster than $O(r^{-2\tau-2})$, and the only term that might decay more slowly is $\widehat{\mathbf{O}}(r^{-\tau-2}) \circledast (g_{\varphi_\varepsilon}^{-1} - g^{-1})$. So far, we can only bound this by $O(r^{-\tau-2})$. However, by applying the same method as in the weighted estimate of the first derivative in Step 2, we can then construct the following barrier function for $D_\beta D_\alpha {\varphi_\varepsilon}$:
\begin{align} \label{baweifunct2}
    u_2 = E' \Big\{ \big( 1- \chi_{\frac{R_0}{2}}  \big) \Big(\frac{R_0}{2} \Big)^{-\tau -2} t(t-1) + \chi_{\frac{R_0}{2}} t(t-1)r^{-\tau-2} \Big\},
\end{align}
where  $R_0$ is the same constant as in \eqref{baweifunct1} and $E'$ is another uniform constant depending on $R_0$, $\|{\varphi_\varepsilon}\|_{C^2(X \times \Sigma,\Theta_\Psi)}$, $\Lambda$, $g$, $J $ and on the constant of \eqref{maasymhigh1}. Hence, we get the weighted estimate for $D_\beta D_\alpha {\varphi_\varepsilon}$:  
\begin{align} \label{maasym2der3}
    |D_\beta D_\alpha {\varphi_\varepsilon}| \leq C\big(\|\varphi_\varepsilon\|_{C^2(X \times \Sigma, \Theta_\Psi)}, \Lambda, g, J, \varepsilon^{-1}\big) r^{-\tau-2}.
\end{align}

According to the full formula \eqref{maasym2der1} for $\Delta_{g_{\varphi_{\varepsilon}}} (D_\beta D_\alpha \varphi_\varepsilon)$ and \eqref{maasymhigh1}, in the restricted unit ball $\hat{B}_{1}(p)$, the $\C^{0,\mathbf{a}}$ norm of all terms on the right hand side of \eqref{maasym2der1} are bounded by $C( \|{\varphi_\varepsilon}\|_{C^2}, \Lambda, g, J, \varepsilon^{-1}) 
 r^{-\tau-2}$. Applying the classic interior and boundary Schauder estimates to \eqref{maasym2der1}, we thus obtain from \eqref{maasym1der3} that 
\begin{align}\label{maasymhigh2}
    \|D_\alpha D_\beta {\varphi_\varepsilon}\|_{\C^{2,\mathbf{a}}(\hat{B}_1(p))} \leq C\big(\|\varphi_\varepsilon\|_{C^2(X \times \Sigma, \Theta_\Psi)}, \Lambda, g, J, \varepsilon^{-1}\big) r^{-\tau-2}.
\end{align}

\subsubsection*{Step 4: Iterative improvement of the barrier estimates.}

In this step, we improve the decay order of the estimates we obtain in \eqref{maasymhigh1} and \eqref{maasymhigh2} by an iteration argument. Recall that from Steps 2--3 we have the following weighted estimates to start the iteration process (see \eqref{maasymhigh1} and \eqref{maasymhigh2}):
\begin{align}
\begin{split} \label{weighted2esti}
\|\nabla_{0,X}\varphi_\varepsilon\|_{\C^{0,\mathbf{a}} (\hat{B}_1 (p))} + \|\nabla_0\nabla_{0,X}\varphi_\varepsilon\|_{\C^{0,\mathbf{a}} (\hat{B}_1 (p))}  & \\ + \; \|\nabla_0^2\nabla_{0,X}\varphi_\varepsilon\|_{\C^{0,\mathbf{a}} (\hat{B}_1 (p))} &= O(r^{-\tau-1}),\\
\|\nabla^2_{0,X} \varphi_{\varepsilon}\|_{\C^{0,\mathbf{a}} (\hat{B}_1 (p))}+\|\nabla_0 \nabla_{0,X}^2\varphi_\varepsilon\|_{\C^{0,\mathbf{a}} (\hat{B}_1 (p))} &= O(r^{-\tau-2}).
\end{split}
\end{align}
To complete the iteration argument, we need to improve the decay of the term $g_{\varphi_\varepsilon}^{-1}- g^{-1}$. More precisely, this term occurs in a combination $(g_{\varphi_\varepsilon}^{ij} - g^{ij}) D_\alpha g_{ij} $ in the first derivative estimate (Step 2), and in combinations $(g_{\varphi_\varepsilon}^{ij} - g^{ij}) D_\beta D_\alpha g_{ij}$ and $ [g_{\varphi_{\varepsilon}}^{ik} D_\beta (g_{\varphi_{\varepsilon}})_{kl} g_{\varphi_{\varepsilon}}^{jl} - g^{ik} D_\beta g_{kl} g^{jl}  ] D_\alpha g_{ij} $ (to get optimal decay rate of $D_\alpha D_\beta 
 \varphi_\varepsilon$, we need to analyze this term) in the second derivative estimate (Step 3). We will now analyze these combinations more carefully. All constants in this step may depend on $\|\varphi_\varepsilon\|_{C^2(X \times \Sigma, \Theta_\Psi)}, \Lambda, g, J, \varepsilon^{-1}$. Let $\varphi$ be a continuous function defined in $(X\backslash B_{R}) \times \Sigma $ with at most polynomial growth rate at infinity,  for simplicity, we introduce the notation $(\varphi)^{\sharp}$ to denote the decay rate of $\varphi$ and $ (D_X \varphi)^\sharp$, $(D_X^2 \varphi )^\sharp$ to denote the decay rate of $\|\nabla_{0,X}\varphi\|_{\C^{0,\mathbf{a}} (\hat{B}_1 (p))} + \|\nabla_0\nabla_{0,X}\varphi\|_{\C^{0,\mathbf{a}} (\hat{B}_1 (p))} + \|\nabla_0^2\nabla_{0,X}\varphi\|_{\C^{0,\mathbf{a}} (\hat{B}_1 (p))} $, $\|\nabla^2_{0,X} \varphi\|_{\C^{0,\mathbf{a}} (\hat{B}_1 (p))}+\|\nabla_0 \nabla_{0,X}^2\varphi\|_{\C^{0,\mathbf{a}} (\hat{B}_1 (p))}$ respectively.

The metric tensor $(g_{\varphi_\varepsilon})_{i j}$ and its inverse can be written as $(2n+2) \times (2n+2)$-matrices
\begin{align*}
    \widetilde{P} = 
    \begin{pmatrix}
    P & \eta^{t} \\
    \eta & \mathfrak{p} 
    \end{pmatrix}, \quad (\widetilde{P})^{-1} = 
    \begin{pmatrix}
    Q & \xi^{t} \\
    \xi & \mathfrak{q}
    \end{pmatrix},
\end{align*}
where $P$, $Q$ are $2n \times 2n$-matrices, $\mathfrak{p}$, $\mathfrak{q}$ are $2 \times 2$-matrices and $\eta$, $\xi$ are $2 \times 2n$-matrices. By direct calculation, we have
\begin{align} \label{lincalc1}
    Q = P^{-1} - P^{-1} \eta^{t} \xi, \quad \xi=- \mathfrak{p}^{-1} \eta Q, \quad  \mathfrak{q} = (I_2 - \xi \eta^{t})\mathfrak{p}^{-1}.
\end{align}
The fact that $ \Lambda^{-1} \varepsilon I_{2n+2} \leq \widetilde{ P} \leq \Lambda I_{2n+2}$ implies that $|\xi| \leq C |\eta|$. The weighted estimate \eqref{weighted2esti}, together with the fall-off condition of the metric $g$, implies that $|\eta| = O(r^{(D_X \varphi_{\varepsilon})^{\sharp}})$. Then, from \eqref{lincalc1}, we have
\begin{align} \label{lincalc2}
    Q= P^{-1} + O(|\eta|^2), \quad \xi = O(|\eta|), \quad \mathfrak{q} = \mathfrak{p}^{-1} + O(|\eta|^2).
\end{align}
Similarly, let $\widetilde{P}'$ denote the matrix of $g$ in asymptotic coordinates. If we write
\begin{align*}
     \widetilde{P}' = 
    \begin{pmatrix}
    P' & (\eta')^{t} \\
    \eta' & \mathfrak{p}' 
    \end{pmatrix}, \quad (\widetilde{P}')^{-1} = 
    \begin{pmatrix}
    Q' & (\xi')^{t} \\
    \xi' & \mathfrak{q}'
    \end{pmatrix},
\end{align*}
then we have
\begin{align} \label{lincalc3}
     Q'= (P')^{-1} + O(|\eta'|^2), \quad \xi' = O(|\eta'|), \quad \mathfrak{q}' = (\mathfrak{p}')^{-1} + O(|\eta'|^2),
\end{align}
where $|\eta'| = O(r^{-\gamma-1})$. According to the estimate \eqref{weighted2esti}, $|P - P'| = O(r^{(D_X^2 \varphi_{\varepsilon})^\sharp})$ and hence $| P^{-1} - (P')^{-1}| = O(r^{(D_X^2 \varphi_{\varepsilon})^\sharp})$ as well because $P,P'$ are uniformly bounded. Moreover, $\mathfrak{p}, \mathfrak{q}, \mathfrak{p}', \mathfrak{q}'$ are all uniformly equivalent to $I_2$ but there is no reason for $\mathfrak{p} - \mathfrak{p}'$ to decay. Then \eqref{lincalc2} and \eqref{lincalc3} imply that
\begin{align} \label{decayest:Qxip}
\begin{split}
  |Q- Q'| &= O(r^{\max\{ (D_X^2 \varphi_{\varepsilon})^\sharp, 2 (D_X \varphi_{\varepsilon})^\sharp, -2\gamma-2 \}} ), \\ 
  |\xi-\xi'| &= O(r^{ \max\{ (D_X \varphi_{\varepsilon})^\sharp, -\gamma -1 \} }),\\
  |\mathfrak{p}- \mathfrak{p}' | &= O(1).
\end{split}
\end{align}
Then, by calculating blockwise and using that $D_{\alpha}g_{\boldsymbol{\mu}\boldsymbol{\nu}} = 0$, we have
\begin{align}\label{eq:foo}
\begin{split}
    ( g_{\varphi_\varepsilon}^{i j} - g^{i j} ) D_\alpha g_{i j}&= |Q- Q'| O( r^{-\tau-1} ) + |\xi -\xi'|  O(r^{-\gamma-2}),\\
     ( g_{\varphi_\varepsilon}^{i j} - g^{i j} ) D_\beta D_\alpha g_{i j}&= |Q- Q'| O( r^{-\tau-2} ) + |\xi -\xi'|  O(r^{-\gamma-3}).
\end{split}
\end{align}
By inserting \eqref{decayest:Qxip}, \eqref{eq:foo} into \eqref{maasym1der2}, we have 
\begin{align} \label{C0iteestimate}
    \begin{split}
        |D_\alpha \varphi_{\varepsilon}| &= O(r^{\max \{(D_X^2 \varphi_{\varepsilon})^\sharp-\tau-1, (D_X \varphi_{\varepsilon})^\sharp -\tau-1, -2\gamma-3, -\varsigma-1\}}).
    \end{split}
\end{align}
For the last but one term of \eqref{maasym2der1}, 
\begin{align} \label{lasttermoflaof2ndder}
\begin{split}
    D_\alpha g_{ij} \big(g_{\varphi_{\varepsilon}}^{ik} g^{jl}_{\varphi_{\varepsilon}} D_\beta (g_{\varphi_{\varepsilon}})_{kl}  - g^{ik} g^{jl} D_\beta g_{kl}  \big) &=  D_\alpha g_{ij} g^{ik}_{\varphi_{\varepsilon}} g^{jl}_{\varphi_{\varepsilon}} (D_\beta (g_{\varphi_{\varepsilon}})_{kl} -D_\beta g_{kl} )  \\
    & + D_{\alpha} g_{ij} D_\beta g_{kl} g_{\varphi_{\varepsilon}}^{jl} (g^{ik}_{\varphi_{\varepsilon}}- g^{ik}) \\
    & + D_{\alpha} g_{ij} D_\beta g_{kl} g^{ik} (g^{jl}_{\varphi_{\varepsilon}}- g^{jl}).
\end{split}
\end{align}
For the first term of right-hand side of \eqref{lasttermoflaof2ndder}, by using $|D_\beta (g_{\varphi_{\varepsilon}})_{kl} - D_\beta g_{kl}| \leq |\nabla_0^2 \nabla_{0,X} \varphi_\varepsilon|$, we obtain that the decay rate of the first term is given by $(D_X \varphi_{\varepsilon})^\sharp -\tau -1$. For the second and third terms, we need to analyze $(g^{ik}_{\varphi_{\varepsilon}} - g^{ik})$. Similar to (\ref{eq:foo}), we have
\begin{align} \label{eq:fooo}
    D_{\alpha} g_{ij} D_\beta g_{kl} g^{ik} (g^{jl}_{\varphi_{\varepsilon}}- g^{jl}) = |Q-Q'| O(r^{-2\tau -2}) + |\xi-\xi'| O(r^{-\tau-\gamma-3}) + O(r^{-2\gamma-4}).
\end{align}
By inserting \eqref{decayest:Qxip} into \eqref{eq:fooo}, we have 
\begin{align}\label{decaylastterm}
\begin{split}
    D_\alpha g_{ij} \big(  g_{\varphi_{\varepsilon}}^{ik} g^{jl}_{\varphi_{\varepsilon}} & D_\beta (g_{\varphi_{\varepsilon}})_{kl}  - g^{ik} g^{jl} D_\beta g_{kl}  \big) = O(r^{\max \{(D_X \varphi_{\varepsilon})^\sharp -\tau-1, (D_X^2 \varphi_{\varepsilon})^\sharp -2\tau -2, -2\gamma-4\}}).
    \end{split}
\end{align}
Then, inserting \eqref{decayest:Qxip}, \eqref{eq:foo} into \eqref{maasym2der1}, we have
\begin{align*}
    |D_\alpha D_\beta \varphi_{\varepsilon}| &= O (r^{\max \{ (D_X^2 \varphi_{\varepsilon})^\sharp-\tau -1, (D_X \varphi_{\varepsilon})^\sharp -\tau-1, -2\gamma-4, -\varsigma-2 \}}). 
\end{align*}

We can go one step further by applying Schauder estimates to \eqref{maasym1der2} and \eqref{maasym2der1} and to obtain $\C^{2,\mathbf{a}}$ estimates for $D_\alpha \varphi_\varepsilon$ and $D_\alpha D_\beta \varphi_\varepsilon$ in $\hat{B}_1 (p)$. Indeed, those terms on the right-hand side of the PDEs \eqref{maasym1der2}, \eqref{maasym2der1} that were known to decay pointwise with rate $\max\{(D_X^2 \varphi_{\varepsilon})^\sharp, (D_X \varphi_{\varepsilon})^\sharp\} -\tau-1$ already after Step 3 are actually also decaying at rate $\max\{(D_X^2 \varphi_{\varepsilon})^\sharp, (D_X \varphi_{\varepsilon})^\sharp\} -\tau-1$ in $\C^{0,\mathbf{a}} (\hat{B}_1 (p))$ norm. 
This is clear from \eqref{weighted2esti}. So we just need to find the decay rates of the most difficult terms, $(g_{\varphi_\varepsilon}^{ij} - g^{ij}) D_\alpha g_{ij} $ in \eqref{maasym1der2} and $(g_{\varphi_\varepsilon}^{ij} - g^{ij}) D_\beta D_\alpha g_{ij}$, $D_\alpha g_{ij} (g_{\varphi_{\varepsilon}}^{ik} g^{jl}_{\varphi_{\varepsilon}} D_\beta (g_{\varphi_{\varepsilon}})_{kl}  - g^{ik} g^{jl} D_\beta g_{kl} )$ in \eqref{maasym2der1} in $\mathcal{C}^{0,\mathbf{a}}(\hat{B}_1(p))$ norm as well. For this, we need to go back and also estimate the $\mathcal{C}^{0,\mathbf{a}}$-norm of $Q-Q'$ and $\xi-\xi'$ in $\hat{B}_1(p)$, as follows. By using \eqref{weighted2esti}, we have that
\begin{align*}
[P^{-1}- (P')^{-1}]_{\mathcal{C}^{0,\mathbf{a}}(\hat{B}_1(p))} = O(r^{(D_X^2 \varphi_\varepsilon)^\sharp} ),\quad [\xi]_{\mathcal{C}^{0,\mathbf{a}}(\hat{B}_1(p))} = O(r^{(D_X \varphi_\varepsilon)^\sharp}).
\end{align*}
Then, based on \eqref{lincalc2}, we have that
\begin{align} \label{weightedestmatrixalpha}
\begin{split}
    [Q-Q']_{\mathcal{C}^{0,\mathbf{a}}(\hat{B}_1(p))} &= O(r^{(D_X^2 \varphi_{\varepsilon})^\sharp, 2 (D_X \varphi_{\varepsilon})^\sharp, -2\gamma-2}), \\ 
    [\xi- \xi']_{\mathcal{C}^{0,\mathbf{a}}(\hat{B}_1(p))} &= O(r^{\max\{ (D_X \varphi_{\varepsilon})^\sharp, -\gamma-1 \}}).
\end{split}
\end{align}
Then we can proceed as in \eqref{eq:foo} and \eqref{decaylastterm}, obtaining that the decay rates of $[\Delta_{\varphi_{\varepsilon}} D_\alpha \varphi_{\varepsilon}]_{\C^{0,\mathbf{a}}(\hat{B}_1 (p))} $ and $[\Delta_{\varphi_{\varepsilon}} D_\alpha D_\beta \varphi_\varepsilon]_{\C^{0,\mathbf{a}}(\hat{B}(p))}$ are $\max \{(D_X^2 \varphi_{\varepsilon})^\sharp-\tau-1, (D_X \varphi_{\varepsilon})^\sharp -\tau-1, -2\gamma-3, -\varsigma-2 \}$ and $\max \{(D_X^2 \varphi_{\varepsilon})^\sharp-\tau-1, (D_X \varphi_{\varepsilon})^\sharp -\tau-1, -2\gamma-4, -\varsigma-2\}$ respectively. According to the classic interior and boundary Schauder estimates, we improve \eqref{C0iteestimate} to $\C^{2,\mathbf{a}} (\hat{B}_1 (p))$ norm,
\begin{align} \label{C2aiteestimate}
    \begin{split}
        \|D_\alpha \varphi_{\varepsilon}\|_{\C^{2,\mathbf{a}} (\hat{B}_1 (p))} &= O(r^{\max \{(D_X^2 \varphi_{\varepsilon})^\sharp-\tau-1, (D_X \varphi_{\varepsilon})^\sharp -\tau-1, -2\gamma-3, -\varsigma-1 \}}),\\
        \|D_\alpha D_\beta \varphi_{\varepsilon}\|_{\C^{2,\mathbf{a}} (\hat{B}_1 (p))} &= O (r^{\max \{ (D_X^2 \varphi_{\varepsilon})^\sharp-\tau -1, (D_X \varphi_{\varepsilon})^\sharp -\tau-1, -2\gamma-4, -\varsigma- 2\}}).
    \end{split}
\end{align}
Inserting \eqref{weighted2esti} into \eqref{C2aiteestimate}, and using \eqref{C2aiteestimate} again to improve \eqref{weighted2esti}, we finally obtain the following estimates:
\begin{align} \label{C2afinalestimate}
        \|D_\alpha \varphi_{\varepsilon}\|_{\C^{2,\mathbf{a}} (\hat{B}_1 (p))} = O(r^{\max\{-2\gamma-3, -\varsigma-1\}}), \quad
        \|D_\alpha D_\beta \varphi_{\varepsilon}\|_{\C^{2,\mathbf{a}} (\hat{B}_1 (p))} = O (r^{ \max\{ -2\gamma-4, -\varsigma-2 \}}).
\end{align}
Note that according to (\ref{C2afinalestimate}), because $\Psi$ was chosen to be linear in $t$, the decay rate of $\varphi_{\varepsilon}$ is faster than the decay rate of the boundary data $\psi_0,\psi_1$. 

\subsubsection*{Step 5: Proof of Theorem \ref{epsilongeodesicasymptotics}}

In Step 4, we have obtained the optimal decay rates in the cases of $k=1,2$ (even though it is not required in the proof of Theorem \ref{epsilongeodesicasymptotics}). In this step, we give optimal estimates for $k\geq 3$ and complete the proof of Theorem \ref{epsilongeodesicasymptotics}.

{For the higher order derivatives, by differentiating the Monge-Amp\`er{e} equation \eqref{maasym1der1} $m$ times, similar to \eqref{maasym1der2} and \eqref{maasym2der1} and writing $D_K = D_{\kappa_1} \cdots D_{\kappa_m}$ ($1 \leq \kappa_i \leq 2n$, for $i =1,\ldots, m$), instead of giving a full formula as \eqref{maasym1der2} and \eqref{maasym2der1}, we write a simplified formula of $\Delta_{\varphi_{\varepsilon}} D_K \varphi_{\varepsilon}$:
\begin{align} \label{higherestfeps1}
\begin{split}
    |\Delta_{\varphi_{\varepsilon}} D_K \varphi_{\varepsilon}|   &  \leq\sum_{i=1}^{m} O(r^{-\tau -2 -m+i} ) |\nabla_{0,X}^i \varphi_{\varepsilon}| + \sum_{i=1}^{m} O(r^{-\tau -1 -m +i}) |\nabla_{0} \nabla_{0,X}^{i} \varphi_{\varepsilon} | \\ & + \sum_{i=1}^{m-1} O(r^{-\tau -m +i}) |\nabla_{0}^2 \nabla_{0,X}^i \varphi_{\varepsilon}|  + \sum_{i=1}^m |\nabla_{0,X}^i g_{jl}| |\nabla_{0,X}^{m-i} (g^{jl}_{\varphi_{\varepsilon}} - g^{jl})|\\ &+ O(r^{-\varsigma-m}). 
\end{split}
\end{align}
  Applying induction on $m$, according to iteration process (step 4), we can assume for $k \leq m-1$
  \begin{align} \label{inducassonm}
  ||\nabla^{k}_{0,X} \varphi_{\varepsilon}||_{\hat{B}_1(p)}= O(r^{\max\{-2\gamma-2, -\varsigma \}-k}).
  \end{align}
To find the optimal decay rates, the most difficult term is $\sum_{i=1}^m |\nabla_{0, X}^i g_{jl}| |\nabla_{0, X}^{m-i} (g^{jl}_{\varphi_{\varepsilon}} - g^{jl})|$. Notice that
by \eqref{eq:fooo} and \eqref{C2aiteestimate}, we have
\begin{align*}
    \big|D_{K_1} {g_{jl}} D_{K_2} g_{ik} (g^{ij}_{\varphi_{\varepsilon}}- g^{ij})\big| = O (r^{-2\gamma-2- k_1-k_2}),
\end{align*}
where $K_1$, $K_2$ are $k_1$-, $k_2$-multi-indices respectively. Then, we apply induction on $k$ to find the decay rate of $|D_{K_1} {g_{jl}} D_{K_2} g_{ik} D_K (g^{ij}_{\varphi_{\varepsilon}}- g^{ij})|$, where $K$ is a $k$-multi-index. Applying one derivative to $(g^{ij}_{\varphi_{\varepsilon}}-g^{ij})$, by \eqref{lasttermoflaof2ndder}, we can prove that 
\begin{align} \label{induconk}
    |D_{K_1} {g_{jl}} D_{K_2} g_{ik} D_K (g^{ij}_{\varphi_{\varepsilon}}- g^{ij})| = O(r^{-2\gamma-2-k_1-k_2-k}).
\end{align}
Then, by \eqref{lasttermoflaof2ndder} and \eqref{induconk}, we have
\begin{align*}
    |\nabla_{0,X}^i g_{jl}| |\nabla_{0,X}^{m-i} (g^{jl}_{\varphi_{\varepsilon}} - g^{jl})|  & \leq |\nabla^i_{0,X} g_{jl}| \Big\{ \big|\nabla_{0,X}^{m-i-1} \big[g^{jk}_{\varphi_{\varepsilon}} g^{sl}_{\varphi_{\varepsilon}} (\nabla_{0,X} (g_{\varphi_{\varepsilon}})_{ks} -\nabla_{0,X} g_{ks})  \big] \big|  \\
    & + 2 \big| \nabla^{m-i-1} \big[ g_{\varphi_{\varepsilon}}^{jk} (g^{ls}_{\varphi_{\varepsilon}}- g^{ik})\nabla_{0,X} g_{ks} \big] \big| \Big\}  \\
    & = O(r^{-2\gamma-2-m}).
\end{align*}
Combining with \eqref{inducassonm}, we have that the right-hand side of \eqref{higherestfeps1} is $O(r^{-2\tau+2 -k})$. Using the construction of barrier functions in Step 2--3, we obtain that $|D_K \varphi_{\varepsilon}| \leq C r^{-2\tau+2 -m}$. To apply Schauder estimates to the $m$-th derivative of Monge-Amp\`{e}re equation, we also need to know the decay rate of $[\Delta_{\varphi_{\varepsilon}} D_K \varphi_{\varepsilon}]_{\C^{0,\mathbf{a}} (\hat{B}(p))} $:
\begin{align*}
    [\Delta_{\varphi_{\varepsilon}} D_K \varphi_{\varepsilon}]_{\C^{0,\mathbf{a}} (\hat{B}(p))} & \leq \sum_{i=1}^{m} O(r^{-\tau -2 -m+i} ) \|\nabla_{0,X}^i \varphi_{\varepsilon}\|_{\C^{0,\mathbf{a}} (\hat{B}(p))} \\ & + \sum_{i=1}^m O(r^{-\tau -1 -m +i}) \|\nabla_{0} \nabla_{0,X}^{i} \varphi_{\varepsilon} \|_{\C^{0,\mathbf{a}} (\hat{B}(p))}  \\ & + \sum_{i=1}^{m-1} O(r^{-\tau -m +i}) \|\nabla_{0}^2 \nabla_{0,X}^i \varphi_{\varepsilon}\|_{\C^{0,\mathbf{a}} (\hat{B}(p))} \\&  + \sum_{i=1}^m\big\||\nabla_{0,X}^i g_{jl}| |\nabla_{0,X}^{m-i} (g^{jl}_{\varphi_{\varepsilon}} - g^{jl})|\big\|_{\C^{0,\mathbf{a}} (\hat{B}(p))}+ O (r^{-\varsigma-m})\\&  = O(r^{{\max\{-2\gamma-2, -\varsigma \}-m}}) 
\end{align*}
Hence, we have $\|D_K \varphi_{\varepsilon}\|_{\C^{2,\mathbf{a}}(\hat{B}_1(p))} \leq C r^{{{\max\{-2\gamma-2, -\varsigma \}-m}}}$, for $m \geq 1$. 
To prove that $\varphi_\varepsilon$ is in $\HH_{-\gamma}$, by integrating $({\varphi_\varepsilon})_r = O(r^{\max\{-2\gamma-2, -\varsigma \}-1})$ in the radial direction from infinity to $r=R$, we obtain a function $\hat{\varphi}_\varepsilon$ defined in $X\setminus B_R$ with decay rate $\max\{-2\gamma-2, -\varsigma \}$. Then,
\begin{align}\label{c0asymbehave}
  \varphi_\varepsilon - \hat{\varphi}_\varepsilon = c(\theta, t),   
\end{align}
where $c(\theta, t)$ is a function in $X \setminus B_R$ independent of radius $r$ and $\theta$ be viewed as a variable on the link. It suffices to prove that $c(\theta, t)$ is independent of $\theta$. By taking derivative of \eqref{c0asymbehave}, we have $|\nabla_{0, X} c(\theta, t)| = O(r^{\max\{-2\gamma-2, -\varsigma \}-1})$. In the case that $c(\theta, t)$ is not constant with respect to $\theta$, $|\nabla_{0,X} c(\theta, t)| \sim r^{-1}$, which contradicts to the fact that $\max\{-2\gamma-2, -\varsigma \} <-1$. Hence we proved that $\varphi_\varepsilon = c(t) + O(r^{-\max\{-2\gamma-2, -\varsigma \}})$. We conclude that, for $\Phi_\varepsilon = \varphi_\varepsilon + \Psi$,
    \begin{align*}
        \sup_{(\RR^{2n}\backslash B_R) \times \Sigma} \left(|\nabla^{k}_{0,X}\Phi_\varepsilon|  + |\nabla^{k}_{0,X} \dot{\Phi}_\varepsilon| +  |\nabla^k_{0,X} \ddot{\Phi}_\varepsilon|\right) \leq C(k,\varepsilon^{-1}) r^{-\gamma-k}\quad\text{for all}\ k \geq 1.
    \end{align*}}
In conclusion, we have proved Theorem \ref{epsilongeodesicasymptotics}.

\section{Convexity of the Mabuchi $K$-energy} \label{secconvK}

According to Theorem \ref{yao2022} (assuming $\tau = \tilde{\tau}$), we can restrict ourselves to the space
\begin{align*}
    \HH_{-2\tau+2} = \{ \varphi \in \hat{\C}^\infty_{-2\tau +2}: \omega_\varphi = \omega + dd^c \varphi >0 \}, \quad \tau > n-1,
\end{align*}
and the function $\upsilon(\varepsilon)$ is constructed by \eqref{coeffEvarepsilon} and \eqref{coeffEvarepsilonf}. 
In the previous section, we proved that for any two given boundary data $\psi_0,\psi_1 \in \HH_{-2\tau+2}$, there exists a solution of the $\varepsilon$-geodesic equation $(E_\varepsilon)$ in the same space $\HH_{-2\tau+2}$.

The derivative of the Mabuchi $K$-energy can be defined as follows: for $\psi \in T_\varphi \HH_{-2\tau +2}$,
\begin{align*}
    \delta_\psi \K (\varphi) = -\int_{X} \psi R(\omega_\varphi) \frac{\omega_\varphi^n}{n!}. 
\end{align*}
The integral converges because $-2-2\tau < -2n$, equivalently, $\tau > n-1$. In the following proposition, the second derivative of Mabuchi $K$-energy will be calculated in $M = X_{R} = \{x \in X: r(x) \leq R \}$ containing boundary terms, and it will be clear that these boundary terms go to zero as $R \to \infty$. Precisely, we consider Mabuchi $K$-energy restricted in $M$,
\begin{align}
    \delta_{\psi} \K_{M} (\varphi) = -\int_{M} \psi R(\omega_{\varphi}) \omega_{\varphi}^n.
\end{align}
The calculation of the second variation of $\K_M$ is due to my advisor Bianca Santoro in one of her unpublished notes, several years before I started this project. The limiting case $R \to \infty$ was previously stated by Aleyasin \cite{aleyasin2014space} without details concerning the vanishing of boundary terms.

To simplify the notation, in the following proposition, we write $R_\varphi = R(\omega_{\varphi})$, $\Ric_{\varphi} = \Ric(\omega_{\varphi})$, $\Delta = g^{i\overline{k}}_{\varphi} \partial_i \partial_{\overline{k}}$, $|\cdot|= |\cdot|_{\omega_\varphi}$, $\nabla= \nabla_{\omega_\varphi}$ and $\mathcal{D} f = \nabla_{i} \nabla_k f dz^i dz^k$, where $\nabla_i \nabla_k f = f_{, ik}$ is a covariant derivative of $f$ with respect to $\omega_\varphi$. Recall that $\mathcal{D}$ is called the Lichnerowicz operator, and $\mathcal{D} f = 0$ if and only if ${\rm grad}^{1,0}f$ is a holomorphic type $(1,0)$ vector field.

\begin{pro}[Santoro]
Along a path of potentials $\varphi(t) \in \HH_{-2\tau+2}$, 
\begin{align}\label{2nddervKbrdy}
\begin{split}
\frac{d^2 \K_M}{dt^2} &=
- \int_M [\ddot \varphi - \frac12 |\nabla \dot \varphi|^2]  R_\varphi \omega_\varphi^n 
+ \int_M |\mathcal{D} \dot \varphi|^2 \omega_\varphi^n \\
& - \frac{n(n-1)}{2}\int_{\partial M} \dot \varphi d^c \dot \varphi \wedge \Ric_{\varphi} \wedge \omega_\varphi^{n-2} 
+ ni \int_{\partial M} \dot \varphi g_\varphi^{k\overline{l}} (\Ric_\varphi)_{i\overline{l}} {\dot \varphi}_k d z^i \wedge \omega_{\varphi}^{n-1} \\
&- ni \int_{\partial M} \dot \varphi g_\varphi^{i\overline{j}}  {\dot \varphi}_{, ik\overline{j}} d z^k \wedge \omega_{\varphi}^{n-1} 
- ni \int_{\partial M} g_\varphi^{k \overline{l}} {\dot \varphi}_{\overline{l}} {\dot \varphi}_{,ki} dz^i \wedge \omega_{\varphi}^{n-1}.
\end{split}
\end{align}
Furthermore, by taking $R \to \infty$ in \eqref{2nddervKbrdy}, we have
\begin{align} \label{2nddervK}
    \frac{d^2 \K}{dt^2} = \int_X \big[\ddot{\varphi} - \frac{1}{2} |\nabla \dot{\varphi}|^2 \big] R_\varphi \frac{\omega_\varphi^n}{n!} + \int_X |\mathcal{D} \dot \varphi|^2 \frac{\omega_\varphi^n}{n!}.
\end{align}
\end{pro}

\begin{proof}
By taking the second derivative of Mabuchi $K$-energy in $M$, we have
\begin{align}
\frac{d^2\mathcal{K}_M}{dt^2} = & \frac{d}{dt} \left[- \int_M R_\varphi\dot\varphi\, \omega_\varphi^n\right] \nonumber
				\\=& - \int_M  \ddot{\varphi} \, R_\varphi \,  \omega_\varphi^n   
				-n \int_M \dot \varphi \frac{d}{dt}(\Ric_\varphi) \wedge \omega_\varphi^{n-1}\label{2dervK1}
 \\ &- n (n-1)  \int_M \dot \varphi \Ric_\varphi \wedge \omega_\varphi^{n-2}\wedge( i\partial \overline{\partial} \dot\varphi). \nonumber
\end{align}
The second term of \eqref{2dervK1} needs one integration by parts, and we get
\begin{align*}
-n \int_M \dot \varphi \frac{d}{dt}(\Ric_\varphi) \wedge \omega_\varphi^{n-1}& =  -n  \int_M \dot \varphi \left[ - i \partial \overline{\partial} \left( \frac{d}{dt}(\log \omega_\varphi^n) \right)\wedge \omega_\varphi^{n-1} \right]  
 \\
 & = 
 n  \int_M \dot \varphi \left[  i \partial \overline{\partial} \left( \frac{n \omega_\varphi^{n-1}\wedge i \partial \overline{\partial} \dot \varphi}{\omega_\varphi^{n}} \right)\wedge \omega_\varphi^{n-1} \right] 
% \\ & = & 
 %2n  \int_M \dot \varphi \left[  i \partial \overline{\partial} \left(\frac12 \De \dot\varphi \right)\wedge \omega_\varphi^{n-1} \right]  +
%\frac{n}{2} \int_{\partial M} \langle \dot\varphi \nabla \De \dot \varphi \\
%+ \De \dot \varphi \nabla \dot \varphi , \vec \nu  \rangle\\
\\ &=
\int_M \dot \varphi (\Delta^2 \dot \varphi )\omega_\varphi^n.
\end{align*}
Now, to the term  $ \int_M \dot \varphi \Ric_\varphi \wedge \omega_\varphi^{n-2}\wedge i\partial \overline{\partial}  \dot\varphi
$. For simplicity, $\dot \varphi = u$, 
\begin{align*}
  &\int_M \dot \varphi \Ric_\varphi \wedge \omega_\varphi^{n-2}\wedge i\partial \overline{\partial} \dot\varphi\\
&=
 -i\int_M \partial u \wedge \overline{\partial} u \wedge \Ric_\varphi \wedge \omega_\varphi^{n-2} + \frac12 \int_{\partial M} u d^c u \wedge \Ric_\varphi \wedge \omega_\varphi^{n-2}  \\
&= -i\int_M \partial u \wedge \overline{\partial} u \wedge \mathring{\Ric}_\varphi \wedge \omega^{n-2}_{\varphi} 
- \frac{i}{n} \int_M \partial u \wedge \overline{\partial} u \wedge R_\varphi \omega^{n-1}_{\varphi} 
\\ & + \frac12 \int_{\partial M} u d^c u \wedge \Ric_\varphi \wedge \omega_\varphi^{n-2}
\\ &= - i\int_M \partial u \wedge \overline{\partial} u \wedge \mathring{\Ric}_\varphi \wedge \omega^{n-2}_{\varphi} 
- \frac{1}{2n^2} \int_M |\nabla u|^2   R_\varphi \omega^{n}_{\varphi} 
\\ & + \frac12 \int_{\partial M} u d^c u \wedge \Ric_\varphi \wedge \omega_\varphi^{n-2},
\end{align*}
where $\mathring{\Ric}$ is the traceless part of Ricci. If $\psi$ is any primitive $(1,1)$-form, then
\begin{align*}
*\psi  = \frac{-1}{(n-2)!} \psi \wedge \omega^{n-2},\;\,\text{and hence}\;\, n(n-1)\mathring{\Ric} \wedge \omega_\varphi^{n-2} = 
-n!({*\mathring{\Ric}}).
\end{align*} 
Hence,
\begin{align*}
&n(n-1) \int_M  i\partial u \wedge \overline{\partial} u \wedge \mathring{\Ric}_\varphi \wedge \omega_\varphi^{n-2} \\
&= 
 -\int_M  n! (*\mathring{\Ric}_{\varphi}) \wedge  (i \partial u \wedge \overline{\partial} u) \\
 & = 
-\int_M   \langle \mathring{\Ric}_{\varphi},  i\partial u \wedge \overline{\partial} u \rangle\, \omega_\varphi^n \\
& = 
 -\int_M   \langle \Ric_{\varphi},  i\partial u \wedge \overline{\partial} u \rangle\, \omega_\varphi^n +
 \int_M \langle \tfrac{1}{n}R_\varphi \,  \omega_\varphi ,i\partial u \wedge \overline{\partial} u \rangle\, \omega_\varphi^n.
\end{align*}
Note that
\begin{align*}
\int_M \langle \tfrac{1}{n}R_{\varphi} \,  \omega_\varphi , i \partial u \wedge \overline{\partial} u \rangle\, \omega_\varphi^n 
&= \int_M \langle (n-1)! R_{\varphi} \,  \omega_\varphi , i \partial u \wedge \overline{\partial} u\rangle\, \frac{\omega_\varphi^n}{n!} \\
&=
\int_M  (i \partial u \wedge \overline{\partial} u)\wedge *[(n-1)! R_{\varphi}  \omega_\varphi]  \\
&=
\int_M R_{\varphi} (i \partial u \wedge \overline{\partial} u) \wedge \omega_\varphi^{n-1}
\\&=
\frac{1}{2n}\int_M  |\nabla u|^2R_{\varphi} \,  \omega_\varphi^{n}.
\end{align*}
Thus, we get that 
\begin{align}\begin{split} 
\frac{d^2 \K_M}{dt^2} 
\label{2ndderivativeofKM}
=&
- \int_M [\ddot \varphi - \frac12 |\nabla \dot \varphi|^2] R_\varphi\, \omega_\varphi^n  - \int_M   \langle \Ric_{\varphi},i \partial u \wedge \overline{\partial} u\rangle\, \omega_\varphi^n \\
& + \int_M u (\Delta^2 u) \, \omega_\varphi^n - \frac{n(n-1)}{2}\int_{\partial M} u d^c u \wedge \Ric_{\varphi} \wedge \omega_\varphi^{n-2}.
\end{split}
\end{align}

\begin{lem} Let $f$ be a smooth function defined on $M$. Then we have that
\begin{align*}
    \Delta^2 f = \mathcal{D}^* \mathcal{D} f - g_{\varphi}^{i\overline{k}} g_{\varphi}^{j\overline{l}} (\Ric_{\varphi})_{i\overline{l}} f_{j \overline{k}}  - g_{\varphi}^{i\overline{k}} g_{\varphi}^{j\overline{l}} (\nabla_{\overline{k}} (\Ric_{\varphi})_{i\overline{l}}) f_j.
\end{align*}
Hence,
\begin{align}\label{intformuoflasqu}
\begin{split}
        \int_M u(\Delta^2 u) \omega_\varphi^n &- \int_{M} \langle \Ric_\varphi, i\partial u \wedge \overline{\partial } u \rangle \omega_\varphi^n \\
        &= \int_M |\mathcal{D} u|^2  \omega_\varphi^n 
    + ni \int_{\partial M} u g_\varphi^{k\overline{l}} (\Ric_\varphi)_{i\overline{l}} u_k d z^i \wedge \omega_{\varphi}^{n-1} 
    \\
    &- ni \int_{\partial M} u g_\varphi^{i\overline{j}} \nabla_{\overline{j}} \nabla_{k} \nabla_{i} u d z^k \wedge \omega_{\varphi}^{n-1} 
    - ni \int_{\partial M} g_\varphi^{k \overline{l}} u_{\overline{l}} u_{,ki} dz^i \wedge \omega_{\varphi}^{n-1}.
\end{split}
\end{align}
\end{lem}
\begin{proof}
Notice that 
\begin{align*}
\nabla_{j}\nabla_{\overline{k}}\nabla_{i} f = \nabla_{\overline{k}}\nabla_{j} \nabla_{i} f - \tensor{R}{_i^m_j_{\overline{k}}} f_m
\end{align*}
Then, we have 
\begin{align*}
    \Delta^2 f &= g^{i\overline{j}}_\varphi g^{k \overline{l}}_\varphi \nabla_{\overline{l}} \nabla_{k}
    \nabla_{\overline{j}}
    \nabla_{i} f \\&= 
    g^{i\overline{j}}_\varphi
    g^{k\overline{l}}_\varphi
    \nabla_{\overline{l}} \big(\nabla_{\overline{j}}
    \nabla_{k} \nabla_{i}
    f - \tensor{R}{_k^m_i_{\overline{j}}} f_m \big)\\ 
    &= \mathcal{D}^* \mathcal{D} f - g_{\varphi}^{k\overline{l}} g_{\varphi}^{m\overline{j}} \Ric_{k\overline{j}} f_{m \overline{l}}  - g_{\varphi}^{k\overline{l}} g_{\varphi}^{m\overline{j}} (\nabla_{\overline{l}} \Ric_{k\overline{j}}) f_m.
\end{align*}
Here $\mathcal{D}^* \mathcal{D} = g^{i\overline{j}} g^{k\overline{l}} \nabla_{\overline{l}} \nabla_{\overline{j}} \nabla_k \nabla_i$. Then, we have
\begin{align*}
    \int_{M} u \mathcal{D}^* \mathcal{D} u \omega_{\varphi}^{n} = \int_{M} g_\varphi^{k\overline{l}} \big( u   g_\varphi^{i\overline{j}} \nabla_{\overline{j}} \nabla_{k} \nabla_{i} u \big)_{\overline{l}} \omega_{\varphi}^n
    - \int_M \big( g_\varphi^{k\overline{l}} u_{\overline{l}} g_\varphi^{i\overline{j}} \nabla_{\overline{j}} \nabla_k \nabla_i u \big) \omega_{\varphi}^n.
\end{align*}
The Stokes' theorem can be applied to the first term in the above formula by observing that if we write $\mathfrak{h} = ih_k dz^k = i(u g_\varphi^{i\overline{j}} \nabla_{\overline{j}} \nabla_{k} \nabla_{i} u) d z^k$, then $g_\varphi^{k\overline{l}} (h_k)_{\overline{l}} \omega_{\varphi}^n = n  \overline{\partial } \mathfrak{h} \wedge \omega_{\varphi}^{n-1}$. Hence, 
\begin{align*}
    \int_{M} g_\varphi^{k\overline{l}} \big( u   g_\varphi^{i\overline{j}} \nabla_{\overline{j}} \nabla_{k} \nabla_{i} u \big)_{\overline{l}} \omega_{\varphi}^n = \int_{\partial M} n \mathfrak{h} \wedge \omega_{\varphi}^{n-1}.
\end{align*}
Similarly,
\begin{align*}
     -\int_{M} \big( g_\varphi^{k\overline{l}} u_{\overline{l}} g_{\varphi}^{i\overline{j}} \nabla_{\overline{j}} \nabla_k \nabla_i u \big) \omega_{\varphi}^n 
     &= - \int_M g_\varphi^{i\overline{j}} \big(g_\varphi^{k\overline{l}} u_{\overline{l}} \nabla_k  \nabla_i u\big)_{\overline{j}} \omega_{\varphi}^n 
     + \int_M |\mathcal{D} u|^2 \omega_{\varphi}^n\\
     & = \int_M |\mathcal{D} u|^2 \omega_{\varphi}^n 
     - ni \int_{\partial M} g_\varphi^{k\overline{l}} u_{\overline{l}} u_{, ki} dz^i \wedge \omega_{\varphi}^{n-1}.
\end{align*}
We have
\begin{align*}
    \int_M u \Delta^2 u 
    &= \int_{M} |\mathcal{D} u|^2 
    - \int_{M} u g_{\varphi}^{k\overline{l}} g_{\varphi}^{m\overline{j}} (\Ric_\varphi)_{k\overline{j}} u_{m \overline{l}}  
    - \int_M u g_{\varphi}^{k\overline{l}} g_{\varphi}^{m\overline{j}} \big(\nabla_{\overline{l}} (\Ric_\varphi)_{k\overline{j}}\big) u_m \\
    &-ni \int_{\partial M} u g_\varphi^{i\overline{j}} \nabla_{\overline{j}} \nabla_{k} \nabla_{i} u d z^k \wedge \omega_{\varphi}^{n-1} - ni \int_{\partial M} g_\varphi^{k \overline{l}} u_{\overline{l}} u_{,ki} dz^i \wedge \omega_{\varphi}^{n-1}.
\end{align*}
Notice that 
\begin{align*}
    -\int_M \langle \Ric_{\varphi}, i \partial u \wedge \overline{\partial} u \rangle \omega_\varphi^n 
    = -\int_M g_\varphi^{i\overline{j}} g_\varphi^{k \overline{l}} (\Ric_\varphi )_{i\overline{l}} u_{k} u_{\overline{j}} \omega_{\varphi}^n,
\end{align*}
and integrating by parts, 
\begin{align*}
    -\int_M g_\varphi^{i\overline{j}} g^{k\overline{l}}_{\varphi} (\Ric_\varphi)_{i\overline{l}} u_k u_{\overline{j}} \omega_\varphi^n 
    &= \int_{M} g^{i\overline{j}}_\varphi \big( g^{k\overline{l}}_\varphi (\Ric_\varphi)_{i\overline{l}} u_k u\big)_{\overline{j}} \omega_\varphi^n 
    + \int_M u g^{i\overline{j}}_\varphi g^{k\overline{l}}_\varphi (\Ric_\varphi)_{i \overline{l}} u_{k\overline{j}}  \omega_\varphi^n \\
    &+ \int_M u g_\varphi^{i\overline{j}} g_\varphi^{k\overline{l}} \big(\nabla_{\overline{j}} (\Ric_{\varphi})_{i\overline{l}} \big) u_k \omega_\varphi^n\\
    & = ni \int_{\partial M} u g_\varphi^{k\overline{l}} (\Ric_\varphi)_{i\overline{l}} u_k d z^i \wedge \omega_{\varphi}^{n-1} 
    + \int_M u g^{i\overline{j}}_\varphi g^{k\overline{l}}_\varphi (\Ric_\varphi)_{i \overline{l}} u_{k\overline{j}}  \omega_\varphi^n \\
    &+ \int_M u g_\varphi^{i\overline{j}} g_\varphi^{k\overline{l}} \big(\nabla_{\overline{j}} (\Ric_{\varphi})_{i\overline{l}} \big) u_k \omega_\varphi^n.
\end{align*}
Hence, we proved that 
\begin{align*}
    \int_M u(\Delta^2 u) \omega_\varphi^n &- \int_{M} \langle \Ric_\varphi, i\partial u \wedge \overline{\partial } u \rangle \omega_\varphi^n 
    = \int_M |\mathcal{D} u|^2  \omega_\varphi^n 
    + ni \int_{\partial M} u g_\varphi^{k\overline{l}} (\Ric_\varphi)_{i\overline{l}} u_k d z^i \wedge \omega_{\varphi}^{n-1} 
    \\
    &- ni \int_{\partial M} u g_\varphi^{i\overline{j}} \nabla_{\overline{j}} \nabla_{k} \nabla_{i} u d z^k \wedge \omega_{\varphi}^{n-1} 
    - ni \int_{\partial M} g_\varphi^{k \overline{l}} u_{\overline{l}} u_{,ki} dz^i \wedge \omega_{\varphi}^{n-1},
\end{align*}
which completes the proof of the lemma.
\end{proof}

The integration formula \eqref{intformuoflasqu} in this lemma, together with  \eqref{2ndderivativeofKM},   completes the proof of (\ref{2nddervKbrdy}).
It suffices to show that all boundary terms in this formula vanish as $R\to\infty$. According to Theorem \ref{epsilongeodesicasymptotics}, we can check that the decay rates of the integrands integrated on $\partial M$ are at most $-2\tau-1 < -2n +1 $. This completes the proof.
\end{proof}

\begin{thm} \label{propconxK}
    Assume that $\omega$ is an ALE K\"ahler metric on $X$ such that the Ricci curvature of $\omega$ is non-positive, $\Ric (\omega) \leq 0$. Then, along each $\varepsilon$-geodesic in $\HH_{-2\tau+2}(\omega)$, $\varphi(t)$, the Mabuchi $K$-energy is convex. 
\end{thm}

\begin{proof}
     The proof is parallel to Chen \cite{chen2000space}. Here, we just do the calculation in the ALE setting. Define $\displaystyle f = \ddot{\varphi} - \frac{1}{2} |\nabla \dot{\varphi}|_{\omega_\varphi}^2$. Then the $\varepsilon$-geodesic equation can be written as
     \begin{align*}
        \varepsilon \frac{\omega^{n}}{\omega^n_{\varphi}} = f.
     \end{align*}
     According to \eqref{2nddervK}, together with the observation, $ \Ric(\omega_\varphi) = \Ric(\omega) + dd^c \log f $, we have
     \begin{align*}
         \frac{d^2\K}{dt^2} &= \int_{X} \big|\DD\dot{\varphi}(t)\big|^2_{\omega_{\varphi}} \omega_{\varphi}^{n} - \int_{X} f R(\omega_\varphi) \, \omega_{\varphi}^n \\
         & = \int_{X} \big|\DD\dot{\varphi}(t)\big|^2_{\omega_{\varphi}} \omega_{\varphi}^{n} - \int_X  f \tr_{\omega_\varphi} \Ric(\omega)\, \omega_\varphi^n - \int_X f \Delta_{\omega_\varphi} \log f \, \omega^n_\varphi \\
         & = \int_{X} \big|\DD\dot{\varphi}(t)\big|^2_{\omega_{\varphi}} \omega_{\varphi}^{n} -\int_X  f \tr_{\omega_\varphi} \Ric(\omega) \, \omega_\varphi^n + \int_X  \frac{|\nabla f|^2_{\omega_\varphi}}{f}\, \omega_\varphi^n\geq 0.
     \end{align*}
     We have the last equality because $f|\nabla \log f|_{\omega_\varphi} = O(r^{-2\tau-1})$ and $-2\tau-1 < -(2n-1)$, so that the relevant boundary integral vanishes. Hence, we have proved the convexity of the Mabuchi $K$-energy.
\end{proof}

\begin{rem} \label{remunique}
    A quick corollary of Theorem \ref{propconxK} is that assuming $\Ric (\omega) \leq 0$, the scalar-flat K\"ahler metric if it exists, is unique in $\HH_{-2\tau+2}(\omega)$. The proof is also parallel to Chen \cite{chen2000space}. However, if there exists a scalar-flat K\"ahler metric in $\HH_{-2\tau +2 } (\omega)$, the condition, $\Ric (\omega) \leq 0$, implies $\Ric (\omega)= 0$. Hence, the uniqueness of scalar-flat ALE K\"ahler metric can be reduced to the uniqueness result of Ricci-flat ALE K\"ahler metric (which can be found in many reference \cite{joyce2000compact, vancoevering2010regularity, conlon2013asymptotically}). A short proof is given as follows. Let $\omega_0$ be a scalar-flat K\"ahler metric in $\HH_{-2\tau +2} (\omega)$. The fact, $\omega_0 = \omega + O(r^{-2\tau})$, implies that the ADM mass of $\omega$ and $\omega_0$ are equal, $\m(\omega) = \m (\omega_1)$. According to the mass formula by Hein-LeBrun \cite{hein2016mass}, 
    \begin{align*}
        \m (\omega) = A(n, c_1(X), [\omega]) + B (n) \int_{X} R(\omega) \frac{\omega^n}{n!},
    \end{align*}
    where $A(n, c_1(X), [\omega])$ is a constant only determined by the dimension $n$, the first Chern class of $X$ and the cohomology class of $\omega$ and $B(n)$ only depends on dimension $n$. The fact, $\m(\omega) = \m (\omega_1)$, together with the mass formula, implies that 
    \begin{align*}
        \int_X R(\omega) = \int_X R (\omega_1) =0.
    \end{align*}
    The  assumption that $\Ric (\omega) \leq 0$ implies that $\Ric(\omega)=0 $. Then, by a simple argument, we can prove that all scalar-flat ALE K\"ahler metrics in $[\omega]$ is actually Ricci-flat. The expansion of scalar-flat K\"ahler metrics (Theorem \ref{yao2022}) implies that the Ricci form, $\Ric (\omega_1)$,  decays to zero at infinite with decay rate faster than  $-2n$. The ddbar lemma implies that there exist $f \in \C^\infty_{2-2n}$ such that 
\begin{align*}
    \Ric (\omega_1) = dd^c f.
\end{align*}
Taking trace with respect to $g$, we have that $\Delta f = 0$. By solving the Laplacian equation (for instance, see \cite[Propsition 2.3]{yao2022mass}), there is a unique solution in the space $\C^{\infty}_{-\delta}$ (for $-\delta \in (-\infty, 0)\backslash D$). Hence, $f\equiv 0$, which implies that $\omega$ is Ricci-flat. 
\end{rem}

\section{Nonexistence of non-positive (or non-negative) Ricci curvature} \label{secnonexist}

Consider the standard family of negative line bundles, $\OO(-k)$, over $\CC\PP^{n-1}$ together with their natural projections $\pi : \OO(-k) \rightarrow \CC\PP^{n-1}$. The total spaces of $\OO(-k)$ are fundamental examples of ALE K\"ahler manifolds by viewing $\OO(-k)$ as a resolution space of $\CC^n/ \ZZ_k$. 
Let $\omega$ be any ALE K\"ahler metric on $\OO(-k)$ asymptotic to the Euclidean metric with decay rate $-\tau$ ($\tau> 0$). In the following, we shall prove the nonexistence of a sign of the Ricci curvature of $\omega$ in the case $k \ne n$. When $k=n$, there always exists a Ricci-flat ALE K\"ahler metric in each compactly supported ALE K\"ahler class, see \cite{Kronheimer1989TheCO,kronheimer1989torelli,van2010ricci}.

\begin{thm}
    Let $\OO(-k)$ be the standard negative line bundle over $\CC\PP^{n-1}$ with $n \geq 2$ and $k \ne n$. Let 
    $\omega$ be an ALE K\"ahler metric on $\OO(-k)$ with decay rate $-\tau$ $(\tau > 0)$. Then, the Ricci form of $\omega$, $\rho$, is of mixed type, i.e., neither $\rho \geq 0$ nor $\rho \leq 0$ is true.
\end{thm}

 \begin{proof} Notice that for each integer $k \geq 1$, there is a compactification of $\OO(-k)$ by adding a divisor at infinity, $D_\infty \cong \CC\PP^{n-1}$. We denote the compactified manifold as $M_{k}$ and the natural embedding $j: \OO(-k) \rightarrow M_{k}$ is holomorphic. $M_k$ is a $\CC\PP^1$-bundle over $\CC\PP^{n-1}$. Denote $D_0$ as the divisor corresponding to the base manifold, $\CC\PP^{n-1} \subset \OO(-k)\hookrightarrow M_k$. Then, the normal line bundles of $D_0$ and $D_\infty$ are given by
 \begin{align}\label{normalbofD}
     N_{D_0/M_k} = \OO(-k), \quad N_{D_\infty/M_k} = \OO(k).
 \end{align}
The following facts on the geometry of $M_k$ can be checked by viewing $M_k$ as a smooth toric variety. $M_k$ can be described by $2n$ coordinate charts with coordinates $\{U_i;\ u_i^1, \ldots ,$ $ u_i^{n-1}, u_i \}$, $\{V_i; \ v_i^1, \ldots, v_i^{n-1}, v_i \}$ ($0 \leq i \leq n-1$), where the coordinates are related by
 \begin{align*}
     (u^1_i, \ldots, u^{n-1}_{i}, u_i) &= \left(\frac{1}{u^i_0}, \frac{u_0^1}{u^i_0}, \ldots, \widehat{\frac{u^i_0}{u^i_0}}, \ldots, \frac{u^{n-1}_0}{u^i_0}, (u^i_0)^k u^n_1 \right), \quad 1 \leq i \leq n-1, \\
     (v^1_i, \ldots, v^{n-1}_i, v_i) &= \left(u^1_i, \ldots, u^{n-1}_i, \frac{1}{u_i}\right), \quad 0 \leq i \leq n-1.
 \end{align*}
 The divisor classes of $M_k$ are generated by the class of $D_0 \cong \CC\PP^{n-1}$, the zero section of $\OO(-k) \subset M_k$, and the class of $D_f$, the total space of the restriction of the $\CC\PP^1$-bundle $M_k \to D_0$ to a linear subspace of $D_0$. Restricting $D_0$ and $D_f$ to $U_0$, we can write
 \begin{align*}
     D_0 = \overline{(u_0 =0)}, \qquad D_f = \overline{(u_0^1 =0)}.
 \end{align*}
 The divisor at infinity, $D_\infty$, can be represented by $(u_0 =\infty) = (v_0 =0) $ and $D_\infty$ can be represented in terms of $D_0$ and $D_{f}$ as follows,
 \begin{align} \label{Dinfty}
     D_\infty = D_0 + k D_f
 \end{align}
 By viewing $D_0$, $D_f$ and $D_\infty$ as smooth complex hypersurfaces of $M_k$, the Poincar\'{e} duals of $D_0$, $D_f$ and $D_{\infty}$ have natural explicit representatives denoted by $\rho_0$, $\rho_f$, $\rho_\infty$ respectively. For instance, in $U_0$, 
 \begin{align}
     \rho_0|_{U_0} &= \frac{1}{n\pi} i \partial\overline{\partial} \log  \frac{(1+ \sum_j|u_0^j|^2)^k |u_0|^2 +1}{(1+\sum_j|u_0^j|^2)^k|u_0|^2}, \label{pdD0} \\
     \rho_f |_{U_0} &= \frac{1}{n\pi} i \partial\overline{\partial} \log \big(1 +\sum_j |u_0^j|^2 \big), \label{pdDf}\\
     \rho_\infty |_{U_0} &= \frac{1}{n\pi} i \partial\overline{\partial} \log \big[\big(1+ \sum_j |u_0^j|^2\big)^k  |u_0|^2 + 1\big]. \label{pdDinfty}
 \end{align}
 
\subsubsection*{Step 1: Extension of the ALE Ricci form to $M_k$.}

Recall that the diffeomorphism $\Phi: (\CC^n)^* /\ZZ_k \rightarrow \OO(-k) \setminus D_0$ gives a holomorphic asymptotic chart of $\OO(-k)$. The diffeomorphism $\Phi$ can be explicitly written as
\begin{align*}
    {\Phi} : (\CC^n)^* \rightarrow \OO(-k)\setminus D_0, \quad {\Phi} (z_1, \ldots, z_n) \big|_{U_0}= \Big( \frac{z_2}{z_1}, \ldots, \frac{z_n}{z_1}, {z_1^k}\Big).
\end{align*}
In the coordinate chart $\{U_0;\ u_0^1, \ldots, u_0^{n-1}, u_0\}$, we have $r^{2k} = (1+ \sum_j|u_0^j|^2)^k |u_0|^2$.
By the asymptotic condition of $\omega$, in an asymptotic chart of $\OO(-k)$,  $\log (\omega^n/ \omega_0^n)$ can be viewed as a function of decay order $O(r^{-\tau})$, where $\omega_0$ is the standard Euclidean metric on the asymptotic chart. Thus, the Ricci form satisfies
\begin{align*}
    \rho = -i\partial \overline{\partial} \log \frac{\omega^{n}}{\omega^n_{0}} = O(r^{-\tau-2}).
\end{align*}
The adjunction formula tells us that as line bundles over $\OO(-k)$,
\begin{align*}
    K_{\OO(-k)} = \frac{n-k}{k}[D_0].
\end{align*}
Since $\rho$ is the curvature form of a Hermitian metric on $K_{\OO(-k)}^{-1}$ and $\rho_0$ is the curvature form of a Hermitian metric on $[D_0]$, it follows that \begin{align*}
    \rho + \frac{n-k}{k}\rho_0 \;\,\text{is globally $i\partial\overline\partial$-exact}.
\end{align*}
By restricting $\rho_0$ in \eqref{pdD0} to the asymptotic chart of $\OO(-k)$, we have
\begin{align*}
    \rho_0 = -i\partial \overline{\partial}  \log (1+ r^{-2k}). 
\end{align*}
Hence, by Theorem \ref{yao2022}, $\rho$ can be written as
\begin{align*}
    \rho = -\frac{n-k}{k} \rho_0 + i\partial\overline{\partial} f \quad \text{ for } f \in \C^\infty_{-\tau'} (\OO(-k)), \quad \tau' = \min\{ 2k, \tau\} > 0.
\end{align*}
Since $\rho$ cannot be extended smoothly to $M_k$, we define a smooth cut-off function $\chi$,
\begin{align*}
    \chi(t) = 
    \begin{cases}
        1, \quad &0 \leq t \leq 1,\\
        0, \quad &t \geq 2, \\
        \text{smooth}, \quad &  1 < t <2,
    \end{cases}
\end{align*}
and we define $\chi_R (t) = \chi(t/R)$. Applying the cutoff function, we can extend $\rho$ to be 
\begin{align*}
    \rho_R = \begin{cases}
       \displaystyle -\frac{n-k}{k} \rho_0 + i \partial \overline{\partial} \big( \chi_R f\big),  \quad & \text{in } M_k \setminus D_\infty,\\
       \vspace{-2.5mm}\\
        \displaystyle -\frac{n-k}{k} \rho_0, \quad & \text{on } D_\infty.
    \end{cases}
\end{align*}

\subsubsection*{Step 2: Integral argument for $n=2$.}

Recall that the intersection numbers between $D_\infty$, $D_0$ and $D_f$ are given by
\begin{align} 
    (D_0)\cdot(D_0) = -k, \quad (D_0)\cdot (D_f) = 1, \quad (D_f)\cdot(D_f) =0, \quad (D_0)\cdot(D_\infty) =0.
\end{align}
In particular, if we integrate $\rho$ over $D_0$, then
\begin{align}\label{intD0}
   \int_{D_0} {\rho} = \int_{D_0} \rho_R = \int_{M_k}  {\rho_R} \wedge \rho_0 = -\frac{2-k}{k} \int_{M_k} \rho_0^2 = 2-k.
\end{align}
On the other hand, we have $\rho_R \to \rho$ pointwise and $\rho_R = O(r^{-\tau'-2})$ uniformly as $R \to \infty$. Hence, by the dominated convergence theorem, 
 \begin{align}\label{intDf}
     \int_{\{u^1_0 =0 \}} \rho =\lim_{R\rightarrow \infty} \int_{\{u^1_0 =0 \}}  \rho_{R} = \int_{M_k}  {\rho_R} \wedge  \rho_f = -\frac{2-k}{k} \int_{M_k} \rho_0 \wedge \rho_f = \frac{k-2}{k}. 
 \end{align}
Now assume that $ \rho$ is seminegative (or semi-positive). Then the left-hand sides of both \eqref{intD0} and \eqref{intDf} are non-positive (or non-negative). However, the right-hand sides have opposite signs because $k \neq 2$. This is a contradiction.

\subsubsection*{Step 3: Integral argument for $n\geq 3$.}

In higher dimension, the difficulty is to calculate the intersection numbers of divisors. However, in the case of $M_k$, we can apply the formula of intersection numbers on toric varieties \cite[Chapter VII.6]{ewald1996combinatorial}, or, more explicitly, take integral of formulas of Poincar\'{e} dual \eqref{pdD0}--\eqref{pdDinfty}. Notice that 
\begin{align}\label{tralala}
    \int_{D_0} \rho_0^{n-1} =  (-k)^{n-1}.
\end{align}
Then, we have
\begin{align*}
    \int_{D_0}{\rho}^{n-1} =\int_{D_0} \rho_R^{n-1} = \int_{M_k} \rho_R^{n-1} \wedge \rho_0 = \Big({-\frac{n-k}{k}} \Big)^{n-1}  (-k)^{n-1}.
\end{align*} 
On the other hand, we have $\rho_R \to \rho$ pointwise and $\rho_R = O(r^{-\tau'-2})$ uniformly as $R \to \infty$. Hence, by the dominated convergence theorem, 
 \begin{align*}\begin{split}
     \int_{\{u^1_0 =0 \}} \rho^{n-1} &=\lim_{R\rightarrow \infty} \int_{\{u^1_0 =0 \}}  \rho_R^{n-1} 
     =\lim_{R \rightarrow \infty} \int_{D_f} \rho_R^{n-1} = \lim_{R \to \infty} \int_{M_k}  \rho_R^{n-1} \wedge  \rho_f \\
     &= \Big( {-\frac{n-k}{k} \Big)^{n-1 }} \int_{M_k} \rho_0^{n-1} \wedge \rho_f = \Big({-\frac{n-k}{k}} \Big)^{n-1} (-k)^{n-2},
     \end{split}
 \end{align*}
where the last equality can be observed from \eqref{Dinfty} and \eqref{tralala}:
\begin{align*}
    \int_{M_k} \rho_0^{n-1} \wedge \rho_f = \int_{M_k} \rho_0^{n-1} \wedge \frac{1}{k}(\rho_\infty - \rho_0) = 0 - \frac{1}{k}\int_{M_k}\rho_0^{n} = (-k)^{n-2} 
\end{align*}
because $\rho_0|_{D_\infty} =0$. By the same argument as in dimension $2$, we complete the proof.
\end{proof}

\bibliographystyle{amsplain}
\bibliography{Reference}

\end{document}